\documentclass{amsart}

\usepackage{color,graphicx,amssymb,latexsym,amsfonts,txfonts,amsmath,amsthm}
\usepackage{pdfsync}
\usepackage{amsmath,amscd}
\usepackage[all,cmtip]{xy}
\usepackage{multicol}

\usepackage{hyperref}
\hypersetup{
    colorlinks=true,       
    linkcolor=blue,          
    citecolor=blue,        
    filecolor=blue,      
    urlcolor=blue           
}

\input epsf

\newtheorem*{theoA}{Theorem}

\newtheorem{theo}{Theorem}

\newtheorem{coro}{Corollary}
\newtheorem{prop}{Proposition}
\newtheorem{lemm}{Lemma}
\newtheorem{fact}{Fact}

\newtheorem{rema}{Remark}
\newtheorem{exam}{Example}

\begin{document}

\title{On Real and Pseudo-Real Rational Maps}
\author{Ruben A. Hidalgo and Sa\'ul Quispe}

\subjclass[2000]{37F10; 37P45}
\keywords{Rational maps, Pseudo-real rational maps, Real rational maps, Automorphisms, Field of moduli, Field of definition}

\address{Departamento de Matem\'atica y Estad\'{\i}stica, Universidad de La Frontera,  Temuco, Chile}
\email{ruben.hidalgo@ufrontera.cl, saul.quispe@ufrontera.cl}
\thanks{Partially supported by Projects Fondecyt 1190001 and 11170129.}

\begin{abstract}
The moduli space ${\rm M}_{d}$, of complex rational maps of degree $d \geq 2$, is a connected complex orbifold which carries a natural real structure, coming from usual complex conjugation. Its real points are the classes of rational maps admitting antiholomorphic automorphisms. The locus of the real points ${\rm M}_{d}({\mathbb R})$ decomposes as a disjoint union of the loci ${\rm M}_{d}^{\mathbb R}$, consisting of the real rational maps, and ${\mathcal P}_{d}$, consisting of the pseudo-real ones. We obtain that, both ${\rm M}_{d}^{\mathbb R}$ and ${\rm M}_{d}({\mathbb R})$, are connected and that  ${\mathcal P}_{d}$ is disconnected. We also observe that 
the group of holomorphic automorphisms of a pseudo-real rational map is either trivial or a cyclic group. For every $n \geq 1$, we construct pseudo-real rational maps whose group of holomorphic automorphisms is cyclic of order $n$. As the field of moduli of a pseudo-real rational map is contained in ${\mathbb R}$, these maps provide examples of rational maps which are not definable over their field of moduli. It seems that these are the only explicit examples in the literature (Silverman) of rational maps which cannot be defined over their field of moduli.
We provide explicit examples of real rational maps which cannot be defined over their field of moduli. Finally, we also observe that every real rational map, which admits a model over the algebraic numbers, can be defined over the real algebraic numbers.   
\end{abstract}

\maketitle

\section{Introduction}
The space ${\rm Rat}_{d}$, of complex rational maps of degree $d \geq 2$, 
 can be identified with the affine variety ${\mathbb P}^{2d+1}({\mathbb C})\setminus V_{d}$, where $V_{d}$ is the Macaulay resultant (a hypersurface of degree $2d$); so it carries the structure of a connected complex manifold of dimension $2d+1$, and  it is known that $\pi_{1}({\rm Rat}_{d}) \cong {\mathbb Z}_{2d}$ \cite{Filom2,Milnor}.
 
For each M\"obius transformation $T \in {\rm PSL}_{2}({\mathbb C})$ and each rational map $\phi \in {\rm Rat}_{d}$, it holds that $\phi^{T}:=T \circ \phi \circ T^{-1} \in  {\rm Rat}_{d}$, and  the induced map $\phi \mapsto \phi^{T}$ is a holomorphic automorphism of ${\rm Rat}_{d}$. 
 The quotient space ${\rm M}_{d}={\rm Rat}_{d}/{\rm PSL}_{2}({\mathbb C})$, the moduli space of complex rational maps of degree $d$, has the structure of a connected complex orbifold of dimension $2(d-1)$ \cite{Milnor2}. If $\phi \in {\rm Rat}_{d}$ and ${\rm Aut}(\phi)$ is its  ${\rm PSL}_{2}({\mathbb C})$-stabilizer (called its group of holomorphic automorphisms), then $[\phi] \in M_{d}$ is a cone point of the orbifold structure of ${\rm M}_{d}$ if and only if ${\rm Aut}(\phi)$ is non-trivial. In \cite{Milnor}, Milnor proved that ${\rm M}_{2}$ can be identified with ${\mathbb C}^{2}$. It is known that 
 $\pi_{1}({\rm M}_{d})=0$ for $d$ even and, for $d$ odd, $\pi_{1}({\rm M}_{d}) \cong {\mathbb Z}_{2}$ \cite{Filom2}. 

Similarly, as for the case of M\"obius transformations, 
if $Q$ is an extended M\"obius transformation (i.e., a composition of a M\"obius transformation with the usual complex conjugation map $J(z)=\overline{z}$) and $\phi \in {\rm Rat}_{d}$, then $\phi^{Q}:=Q \circ \phi \circ Q^{-1} \in {\rm Rat}_{d}$ (if $\phi^{Q}=\phi$, then $Q$ is called an antiholomorphic automorphism of $\phi$) and the induced map $\phi \mapsto \phi^{Q}$ is an antiholomorphic automorphism of ${\rm Rat}_{d}$. 
Each of these antiholomorphic automorphisms induces the same antiholomorphic automorphism of order two $\widehat{J}$ on ${\rm M}_{d}$ (i.e., a real structure on ${\rm M}_{d}$). A fixed point of $\widehat{J}$ is called a real point of ${\rm M}_{d}$ and the set of all the real points is denoted by ${\rm M}_{d}({\mathbb R}) \subset {\rm M}_{d}$. 

If $\phi \in {\rm Rat}_{d}$, then $[\phi] \in {\rm M}_{d}({\mathbb R})$ if and only if $\phi$ admits an antiholomorphic automorphism (Lemma \ref{realista}). 
If a rational map admits a reflection (i.e., antiholomorphic involution with fixed points) as an automorphism, then it is called a real rational map; otherwise, it is called pseudo-real. If we denote by ${\rm M}_{d}^{\mathbb R}$ (respectively, ${\mathcal P}_{d}$) the locus in ${\rm M}_{d}$ consisting of the equivalence classes of real (respectively, pseudo-real) rational maps, then  ${\rm M}_{d}({\mathbb R})$ is the disjoint union of ${\rm M}_{d}^{\mathbb R}$ and ${\mathcal P}_{d}$. It is known that ${\rm M}_{d}^{\mathbb R}$ is a connected (Lemma \ref{Md(R)conexo}) real orbifold of real dimension $2(d-1)$ \cite{Filom,Silv2} and, for $d \geq 2$ even, Silverman \cite{Silv1} proved that ${\mathcal P}_{d}=\emptyset$. 
We prove the following connectedness properties.

\begin{theoA}
The space ${\rm M}_{d}({\mathbb R})$ is connected and, for $d \geq 3$ odd,  
${\mathcal P}_{d}$ is disconnected.
\end{theoA}

The proof of the first part is done at the end of Section \ref{Sec:auto} and the second one is provided in Section \ref{Sec:teo6}.

For each integer $n \geq 1$, we let 
${\mathcal B}_{d}(n) \subset {\rm M}_{d}$ be the locus of equivalence classes of degree $d$ rational maps admitting an antiholomorphic automorphism of order $2n$ (where, for $n=1$, we assume the automorphism is an imaginary reflection).  A characterization of the elements $\phi \in {\rm Rat}_{d}$ such that $[\phi] \in {\mathcal B}_{d}(n)$ is provided (Theorem \ref{teo9}), from which we may see that 
${\mathcal B}^{\mathbb R}_{d}(n):={\mathcal B}_{d}(n) \cap {\rm M}_{d}^{\mathbb R} \neq \emptyset$. We prove that ${\mathcal B}_{d}(n)$ is a connected real orbifold and we also compute its real dimension (Theorems \ref{teoimaginary} and \ref{conexon=1} and Corollaries  \ref{teoidentifica} and \ref{conexidad}).  The above permits to obtain the connectedness of ${\rm M}_{d}({\mathbb R})$ (Corollary \ref{realconexo}). The elements in ${\mathcal P}_{d}(n):={\mathcal P}_{d} \cap {\mathcal B}_{d}(n)={\mathcal B}_{d}(n) \setminus {\mathcal B}^{\mathbb R}_{d}(n)$ are exactly the classes of pseudo-real rational maps admitting antiholomorphic automorphisms of order $2n$.

In Section \ref{Sec:main} we provide the following fact on the group of automorphisms of pseudo-real rational maps.

\begin{theoA}
If $\phi \in {\rm Rat}_{d}$ is pseudo-real, then ${\rm Aut}(\phi)$  is either trivial or a cyclic group.
\end{theoA}

A characterization of pseudo-real rational maps,  in terms of their group of holomorphic automorphisms, is provided (Theorems \ref{autotrivial} and \ref{teo7}). 
Silverman's examples \cite{Silv1} are pseudo-real rational maps with trivial group of holomorphic automorphisms.
For each $n \geq 1$, we show the existence of a pseudo-real rational map whose group of holomorphic automorphisms is the cyclic group of order $n$ (Corollary \ref{coron=2} and Theorems \ref{T4} and \ref{teo10}). 
We may note that ${\mathcal P}_{d}$ is a union of the sets ${\mathcal P}_{d}(n)$, for $n \geq 1$. We observe that ${\mathcal P}_{d}(1) \neq \emptyset$ (Corollary \ref{coro3}) and that, if $n=2^{s}q$, where $q$ is odd, then ${\mathcal P}_{d}(n) \subset {\mathcal P}_{d}(2^{s})$. Then we note that ${\mathcal P}_{d}(2) \neq \emptyset$ (Corollary \ref{coron=2}) and that ${\mathcal P}_{d}(1) \cap {\mathcal P}_{d}(2^{s}) = \emptyset$ (Proposition \ref{interseccionreal}).
The above permits us to obtain that, for $d \geq 3$ odd, the set ${\mathcal P}_{d}$ is disconnected (Theorem \ref{disconexo}).

The second part of this article is more arithmetic in nature, it concerns with fields of definition and fields of moduli of rational maps.
A field of definition of a rational map is a subfield of ${\mathbb C}$ such that there is an equivalent rational map whose coefficients belong to that subfield.
The moduli space ${\rm M}_{d}$ has the structure of an affine algebraic variety defined over ${\mathbb Z}$ \cite{Silv2}. The field of definition of the point $[\phi]$ of the algebraic variety ${\rm M}_{d}$ is called the field of moduli of $\phi$; it coincides with the intersection of all its fields of definition \cite{Koizumi}.  In general, the field of moduli is not a field of definition; cohomological obstructions were obtained by Silverman in \cite{Silv1}.
Pseudo-real rational maps are the simplest examples of rational maps which cannot be defined over its field of moduli \cite{Silv1}. In the case of a real rational map, it is not so straightforward to decide if it can or not be definable over its field of moduli. By a suitable modification of Silverman's examples, we construct explicit real rational maps which cannot be defined over their field of moduli. In Section \ref{Sec:real} we obtain the following.

\begin{theoA}
If $d \geq 3$ is an odd integer, then 
the rational map $$\phi(z)=\left(3+2\sqrt{2}\;\right) \left( \frac{z^{d}+1}{z^{d}-1} \right) \in {\rm Rat}_{d}\left({\mathbb Q}\left(\sqrt{2} \;\right)\right)$$ 
has ${\mathbb Q}$ as its field of moduli but cannot be defined over it.
\end{theoA}

Finally, in Section \ref{Sec:arit} we obtain the following.

\begin{theoA}
Let ${\mathcal Q}$ be a conjugate-invariant algebraically closed subfield of ${\mathbb C}$.
A real rational map which is definable over ${\mathcal Q}$ is also definable over ${\mathbb R} \cap {\mathcal Q}$.
\end{theoA} 

A typical example of the above is when ${\mathcal Q}=\overline{\mathbb Q}$,
the field of algebraic numbers (i.e., arithmetic rational maps). In

\begin{rema}
As a final remark, let us note that the rational maps whose classes belong to ${\mathcal B}_{d}(1)$ are the ones that commute with the antiholomorphic transformation $\tau_{1}(z)=-1/\overline{z}$, so each of them induce dianalytic maps of the real projective plane ${\mathbb P}^{2}({\mathbb R})=\widehat{\mathbb C}/\langle \tau_{1}\rangle$ (and conversely). This type of maps have been considered in several papers. In \cite{BBM} and \cite[\S 5]{M}, the authors study the dynamics of elements in ${\mathcal B}_{3}(1)$.
\end{rema}

\section{Preliminaries}
We will use the language of orbifols and quotient spaces. Good references on these objects are \cite{Neeman,Thurston1,Thurston2}. For details on M\"obius and extended M\"obius transformations, the reader may consult the books \cite{Beardon} and \cite{Maskit:book}.

\subsection{Holomorphic and antiholomorphic transformations}

\subsubsection{M\"obius and extended M\"obius transformations}
It is well known that the holomorphic automorphisms of the Riemann sphere $\widehat{\mathbb C}$ are provided by the M\"obius transformations, that is, elements of the form
$A(z)=(az+b)/(cz+d)$, where $a,b,c,d \in {\mathbb C}$ are such that $ad-bc=1$. The group of M\"obius transformations can be identified with the projective linear group ${\rm PSL}_{2}({\mathbb C})$. The antiholomorphic automorphisms of $\widehat{\mathbb C}$ are provided by the extended M\"obius transformations (also called fractional reflections in \cite{Maskit:book}), that is, elements of the form $B(z)=(a\overline{z}+b)/(c\overline{z}+d)$, where $a,b,c,d \in {\mathbb C}$ are such that $ad-bc=1$.
The full group $\widehat{\rm PSL}_{2}({\mathbb C})$ of holomorphic and antiholomorphic automorphisms of $\widehat{\mathbb C}$  is generated by ${\rm PSL}_{2}({\mathbb C})$ and the reflection $J(z)=\overline{z}$.

\subsubsection{Classification of M\"obius transformations}
The M\"obius transformations are classified into elliptic, parabolic and loxodromic transformations. The elliptic ones can be conjugated to elements of the form $E(z)=e^{i \theta}z$, where $\theta \in {\mathbb R}$, the parabolic ones can be conjugated to $P(z)=z+1$, and the loxodromic ones can be conjugated to $L(z)=\lambda z$, where $\lambda \in {\mathbb C}$, $|\lambda| \neq 0,1$. The only elements of finite order are provided by the elliptic ones which can be conjugated to a power of one of the form $F(z)=e^{2 \pi i/n}z$. The finite subgroups of ${\rm PSL}_{2}({\mathbb C})$ are either cyclic, dihedral or one of the Platonic groups ${\mathcal A}_{4}$, ${\mathcal A}_{5}$ or ${\mathfrak S}_{4}$ \cite{Beardon}.

\subsubsection{Classification of extended M\"obius transformations}
The extended M\"obius transformations are classified as follows. A reflection is a transformation with a circle of fixed points, an imaginary reflection is one of order two without fixed points, a semi-elliptic is one whose square is an elliptic transformation, a semi-parabolic is one whose square is a parabolic transformation and a semi-hyperbolic is one whose square is a loxodromic transformation.
A reflection can be conjugated to $J(z)=\overline{z}$, an imaginary reflection can be conjugated to $\tau_{1}(z)=-1/\overline{z}$ and a semi-elliptic one can be conjugated to one of the form $\rho(z)=k/\overline{z}$, where $k$ is real and negative \cite[I.E.]{Maskit:book}. In particular, the extended M\"obius transformations of finite order are either reflections, imaginary reflections and those conjugated to odd powers of $\tau_{n}(z)=e^{i\pi/n}/\overline{z}$, where $n \geq 1$. Note that the only extended M\"obius transformations of finite order with fixed points are the reflections.

\subsubsection{An auxiliary lemma}
We will make use of the following fact in the proof of Proposition \ref{interseccionreal}.

\begin{lemm}\label{lema1}
Let $\eta_{1}$ and $\eta_{2}$ two different semi-elliptic transformations such that $G=\langle \eta_{1}, \eta_{2}\rangle$ is a finite group. Let $G^{+}$ be the index two subgroup of $G$ formed by its holomorphic elements. If $G^{+}$ is not cyclic, then $G$ contains a reflection.
\end{lemm}
\begin{proof}
If either $\eta_{1}$ or $\eta_{2}$ is a reflection, then we are done. Let assume none of them is a reflection. 
Let $G^{+}$ the index two subgroup of $G$ consisting of the M\"obius transformations. As we are assuming $\eta_{1}$ and $\eta_{2}$ to be different, this group is not the trivial one. As we are also assuming $G^{+}$ is not a cyclic group, it follows that $G^{+}$ is either dihedral or one of the Platonic groups ${\mathcal A}_{4}$, ${\mathcal A}_{5}$ or ${\mathfrak S}_{4}$. 
By the uniformization theorem, we may assume that the underlying Riemann surface of the quotient orbifold $\widehat{\mathbb C}/G^{+}$ 
is $\widehat{\mathbb C}$. Let us consider a regular branched covering $\pi:\widehat{\mathbb C} \to \widehat{\mathbb C}$ whose deck group is $G^{+}$. Its branch value set ${\mathcal B}$ has cardinality $3$. The restriction $\pi_{0}:\widehat{\mathbb C} \setminus \pi^{-1}({\mathcal B}) \to \widehat{\mathbb C} \setminus {\mathcal B}$ turns out to be a regular covering map with deck group $G^{+}$.
As $G^{+}$ is invariant, under conjugation, by $\eta_{j}$, and $\eta_{j}^{2} \in G^{+}$, we have that $\eta_{j}$ induces an antiholomorphic involution $\eta$ of $\widehat{\mathbb C}$ (preserving the set ${\mathcal B}$) such that $\pi \circ \eta_{j} = \eta \circ \pi$.
 As ${\mathcal B}$ has cardinality three, $\eta$ must fix one of them. This means that $\eta$ is a reflection, that is, it has a circle $C$ of fixed points. Let $\Sigma=C \setminus {\mathcal B}$. As $\pi_{0}$ is a regular covering map and $\eta$ is induced by elements of $G$ (which normalizes the deck group $G^{+}$), there must be an element of $G$ with fixed points, so necessarily a reflection.
\end{proof}

\subsection{The moduli space of rational maps and its branch locus}

\subsubsection{The space of rational maps}
The space ${\rm Rat}_{d}$, of complex rational maps of degree $d \geq 2$, can be identified with the Zariski open set ${\mathbb P}^{2d+1}({\mathbb C})\setminus {\rm Res}_{d}$, where ${\rm Res}_{d} \subset {\mathbb P}^{2d+1}({\mathbb C})$ denotes the algebraic hypersurface defined by the resultant of two polynomials of degree at most $d$, under the injective map 
$$\Xi: {\rm Rat}_{d} \hookrightarrow {\mathbb P}^{2d+1}({\mathbb C}):
\phi(z)=\frac{\sum_{k=0}^{d} a_{k}z^{k}}{\sum_{k=0}^{d} b_{k}z^{k}} \mapsto \Xi(\phi)=[a_{0}:\cdots:a_{d}:b_{0}:\cdots:b_{d}].$$

In particular (see, for instance, \cite{Beardon}),  ${\rm Rat}_{d}$ carries a complex manifold structure of dimension $2d+1$. As ${\rm Res}_{d}$ has real codimension two in ${\mathbb P}^{2d+1}({\mathbb C})$, it follows that ${\rm Rat}_{d}$ is connected (see, for instance, \cite[Corollary 0.3.11]{DV}).  In \cite{Filom}, it was noted that 
$X_{d}:=\{(P,Q) \in {\mathbb C}[z]^{2}: {\rm max}\{{\rm deg}(P),{\rm deg}(Q)\}=d, {\rm Res}_{d}(P,Q)=1\}$ is a 
universal cover of ${\rm Rat}_{d}$, with universal covering map $(P,Q) \in X_{d} \mapsto P/Q \in {\rm Rat}_{d}$, whose deck group is $\pi_{1}({\rm Rat}_{d}) \cong {\mathbb Z}_{2d}$.

\subsubsection{Holomorphic automorphisms of ${\rm Rat}_{d}$}
If $T \in {\rm PSL}_{2}({\mathbb C})$ and $\phi \in {\rm Rat}_{d}$, then $\phi^{T}=T \circ \phi \circ T^{-1} \in {\rm Rat}_{d}$, and this induces a holomorphic automorphism of ${\rm Rat}_{d}$ (the restriction of a linear automorphism of ${\mathbb P}^{2d+1}({\mathbb C})$) \cite{Beardon2}. In this way, 
${\rm PSL}_{2}({\mathbb C})$ acts  as a group of holomorphic automorphisms of ${\rm Rat}_{d}$. 

The group ${\rm Aut}(\phi)=\{T \in {\rm PSL}_{2}({\mathbb C}): \phi^{T}=\phi\}$ (i.e., the ${\rm PSL}_{2}({\mathbb C})$-stabilizer of $\phi$) is called the group of holomorphic automorphisms of $\phi$. This group is finite (since a M\"obius transformation is uniquely determined by its action at three different points). It is known that a non-trivial finite subgroup of ${\rm PSL}_{2}({\mathbb C})$ is isomorphic to either: a cyclic group, a dihedral group, the alternating groups ${\mathcal A}_{4}, {\mathcal A}_{5}$ or the symmetric group ${\mathfrak S}_{4}$ \cite{Beardon}. In \cite{MSW}, it has been observed that each of these finite groups can be realized as the full group of holomorphic automorphisms of a rational map.

\subsubsection{The moduli space of rational maps}
The quotient space ${\rm M}_{d}={\rm Rat}_{d}/{\rm PSL}_{2}({\mathbb C})$ is called the moduli space of rational maps of degree $d$. 
In \cite{Milnor}, Milnor proved that ${\rm M}_{d}$ carries the structure of a connected complex orbifold of dimension $2(d-1)$
and, in \cite{Silv2}, Silverman showed that it has a natural structure of an affine geometric quotient defined over ${\mathbb Q}$.  
Milnor also noted that ${\rm M}_{2}$ can be identified with ${\mathbb C}^{2}$ (by using two of the three symmetrical forms of the multipliers of the fixed points). Recently, West \cite{West} obtained an explicit model for ${\rm M}_{3}$. It seems that, for $d \geq 4$, no explicit model for ${\rm M}_{d}$ is known. It is known that  
 $\pi_{1}({\rm M}_{d})=0$ for $d$ even and, for $d$ odd, $\pi_{1}({\rm M}_{d}) \cong {\mathbb Z}_{2}$ \cite{Filom2}.

\subsubsection{The branch locus of moduli space}
The branch locus of ${\rm M}_{d}$,  ${\mathcal B}_{d}=\{[\phi] \in {\rm M}_{d}: {\rm Aut}(\phi) \neq \{I\}\}$, was studied by Milnor in \cite{Milnor}. By identifying ${\rm M}_{2}$ with ${\mathbb C}^{2}$, the branch locus ${\mathcal B}_{2} \subset {\mathbb C}^{2}$ corresponds to the cubic curve \cite{Fujimura}
$ 2 x^3 + x^2 y - x^2- 4 y^2 - 8 x y + 12 x + 12 y  -36=0,$
where the cuspid $(-6,12)$ corresponds to the class of the rational map $\phi(z)=1/z^{2}$ with ${\rm Aut}(\phi) \cong D_{3}$ (the dihedral group of order $6$); all other points in the cubic correspond to those rational maps with the cyclic group ${\mathbb Z}_{2}$ as full group of holomorphic automorphisms. The above, in particular, permits to observe that ${\mathcal B}_{2} \neq \emptyset$ is connected. For $d \geq 3$,  it is known that ${\mathcal B}_{d} \neq \emptyset$ is the locus where ${\rm M}_{d}$ fails to be a topological manifold \cite{MSW} and, in \cite{HQ}, it was observed that ${\mathcal B}_{d}$ is always connected.

\subsection{The real structure of moduli space}
\subsubsection{Antiholomorhic automorphisms of ${\rm Rat}_{d}$}
Let us consider an extended M\"obius transformation $Q=T \circ J$, where $T \in {\rm PSL_{2}}({\mathbb C})$ and $J(z)=\overline{z}$. 
If  $\phi \in {\rm Rat}_{d}$, then $\phi^{Q}=Q \circ \phi \circ Q^{-1} \in {\rm Rat}_{d}$. This provides an antiholomorphic automorphism of ${\rm Rat}_{d}$
$$\phi \in {\rm Rat}_{d} \mapsto \phi^{Q} \in {\rm Rat}_{d}.$$

For instance, the antiholomorphic automorphism of ${\rm Rat}_{d}$, induced by the reflection $J$, corresponds to the canonical conjugation in the model $\Xi({\rm Rat}_{d})={\mathbb P}^{2d+1}({\mathbb C}) \setminus {\rm Res}_{d}$:
\begin{equation}\label{conjugacion}
\begin{array}{ccc}
\phi \in {\rm Rat}_{d}& \mapsto & [a_{0}:\cdots:b_{d}] \in \Xi({\rm Rat}_{d}) \subset {\mathbb P}^{2d+1}({\mathbb C})\\
\downarrow & & \downarrow \\
\overline{\phi}=\phi^{J} \in {\rm Rat}_{d} & \mapsto & [\overline{a_{0}}:\cdots:\overline{b_{d}}]  \in \Xi({\rm Rat}_{d})  \subset {\mathbb P}^{2d+1}({\mathbb C}).
\end{array}
\end{equation}

Note that: (i) $\overline{\phi}=\phi^{J}$ is obtained from $\phi$ after conjugating its coefficients and (ii) the antiholomorphic automorphism induced by $Q$ is 
the composition of the above usual complex conjugation in the projective space with the restriction of the linear automorphism induced by $T$.

The $\widehat{\rm PSL}_{2}({\mathbb C})$-stabilizer of $\phi \in {\rm Rat}_{d}$, denoted by $\widehat{\rm Aut}(\phi)$, is the group of holomorphic and antiholomorphic automorphisms of $\phi$; the elements in $\widehat{\rm Aut}(\phi) \setminus {\rm Aut}(\phi)$ are the antiholomorphic automorphisms of $\phi$ (generically, a rational map of degree $d \geq 2$ has no antiholomorphic automorphisms).  As ${\rm Aut}(\phi)$ is a subgroup of index at most two in $\widehat{\rm Aut}(\phi)$, for $d \geq 2$, this group is also finite.

\subsubsection{The real structure on moduli space}
Any two extended M\"obius transformations induce the same antiholomorphic involution (i.e., a real structure) $\widehat{J}\;$ on ${\rm M}_{d}$; its fixed points are called its real points and the set of these real points is the real locus ${\rm M}_{d}({\mathbb R})$. A description of the real locus, in terms of automorphisms, is provided by the following.

\begin{lemm}\label{realista}
A point $[\phi] \in {\rm M}_{d}$ belongs to ${\rm M}_{d}({\mathbb R})$ if and only if $\widehat{\rm Aut}(\phi) \setminus {\rm Aut}(\phi) \neq \emptyset$.
\end{lemm}
\begin{proof}
Let us observe that $[\phi] \in {\rm M}_{d}$ is a real point if and only if $\phi \sim \overline{\phi}$, equivalently, if and only if there is  a M\"obius transformation $T \in {\rm PSL}_{2}({\mathbb C})$ such that $J \circ \phi \circ J=\overline{\phi}=T \circ \phi \circ T^{-1}$. This is equivalent that $J \circ T$ is an antiholomorphic automorphism of $\phi$. 
\end{proof}

\subsubsection{Real rational maps}
Let us consider the antiholomorphic automorphism of order two (induced by the complex conjugation map) as described in (\ref{conjugacion}) above, and let 
${\rm Rat}_{d}({\mathbb R}) \subset {\rm Rat}_{d}$ be the set  of its fixed points. Note that every rational map in  ${\rm Rat}_{d}({\mathbb R})$ is one 
whose coefficients belong to ${\mathbb R}$, in particular, they admit the reflection $J(z)=\overline{z}$ as an antiholomorphic automorphism.  Note also that ${\rm Rat}_{d}(\mathbb{R})$ is a real analytic submanifold of ${\rm Rat}_{d}$ of real dimension $2d+1$ (an open dense subset of ${\mathbb P}^{2d+1}({\mathbb R})$).
A rational map $\phi \in {\rm Rat}_{d}$ will be called a real rational map if it is equivalent to one in ${\rm Rat}_{d}({\mathbb R})$ (this is equivalent to say that $\phi$  
admits a reflection as antiholomorphic automorphism \cite{Hidalgo}).
We denote by ${\rm M}_{d}^{\mathbb R} \subset {\rm M}_{d}$ the locus consisting of the equivalence classes of real rational maps. Note that ${\rm M}_{d}^{\mathbb R} \subset {\rm M}_{d}({\mathbb R})$.

\subsubsection{Pseudo-real rational maps}
A rational map $\phi \in {\rm Rat}_{d}$ is called a pseudo-real rational map if it admits antiholomorphic automorphisms, but none of them is a reflection.
We denote by ${\mathcal P}_{d} \subset {\rm M}_{d}$ the locus consisting of the equivalence classes of pseudo-real rational maps. Again, ${\mathcal P}_{d} \subset {\rm M}_{d}({\mathbb R})$. In \cite{Silv1}, Silverman proved that, for $d \geq 2$ even, ${\mathcal P}_{d}=\emptyset$.

\subsection{On the connectedness of ${\rm M}_{d}^{\mathbb R}$}
By Lemma \ref{realista},  ${\rm M}_{d}({\mathbb R})={\rm M}_{d}^{\mathbb R} \sqcup {\mathcal P}_{d}$. In this section we provide an argument to show the connectedness of ${\rm M}_{d}^{\mathbb R}$. 

A {\it real movement} of a real rational map $R \in {\rm Rat}_{d}({\mathbb R})$ is a continuous path $\gamma:[0,1] \to {\rm Rat}_{d}: t \mapsto R_{t}$, where each $R_{t} \in {\rm Rat}_{d}({\mathbb R})$ and $R=R_{0}$. Note that in such a real movement, we have the following properties:
\begin{enumerate}
\item[(i)] different zeroes (respectively, poles) may collide, but we cannot collide a zero with a pole (as we must keep the degree $d$ fixed), and 
\item[(ii)] the zeroes (respectively, the poles) of each $R_{t}$ are either reals or they are comming in pairs of complex conjugated points (with the same multiplicity). 
\end{enumerate}

The above real movement induces a continuous path in ${\rm M}_{d}^{\mathbb R}$ connecting $[R]$ and $[R_{1}]$. 

\begin{lemm}\label{Md(R)conexo}
If $d \geq 2$, then ${\rm M}_{d}^{\mathbb R}$ is connected.
\end{lemm}
\begin{proof}
The idea is to observe that, by performing real movements, we may connect any real rational map of degree $d \geq 2$ to a real rational map equivalent to $z \mapsto z^{d}$.

Let us start with a rational map $R(z)=\frac{P(z)}{Q(z)} \in {\rm Rat}_{d}$, where $P,Q \in {\mathbb R}[z]$ are relatively prime, and $d \geq 2$. Up to a real movement, we may assume all of its zeroes and poles to be contained in ${\mathbb R}$. In this case, both $P$ and $Q$ has the same degree $d$. So, 
\begin{equation}\label{eqrat}
R(z)=\lambda \frac{\prod_{j=1}^{r}(z-a_{j})^{l_{j}} }
{\prod_{j=1}^{s}(z-b_{j})^{t_{j}} },
\end{equation}
where $\lambda \in {\mathbb R} \setminus\{0\}$, $a_{j},b_{j} \in {\mathbb R}$ are different, and 
$d=l_{1}+\cdots+l_{r}=t_{1}+\cdots+t_{s}$.

\begin{fact}\label{Fact1}
We may assume that $r=s=L+1$, for some $L \geq 0$, such that 
$${\mathcal F}=\{a_{1}<b_{1}<a_{2}< \cdots <a_{L+1}<b_{L+1}\}.$$
\end{fact}

\begin{proof}
If two real zeroes $a_{j}$ and $a_{i}$ (respectively, two real poles $b_{j}$ and $b_{i}$) are such that between them there is no another zero or pole, then we perform a real movement that collide these two zeroes (respectively, the two poles), that is, the product $(z-a_{j})^{l_{j}}(z-a_{i})^{l_{i}}$ (respectively, $(z-b_{j})^{t_{j}}(z-b_{i})^{t_{i}}$) is transformed to $(z-a_{j})^{l_{j}+l_{i}}$ (respectively, $(z-b_{j})^{t_{j}+t_{i}}$). 
After this type of real movements, we may assume that the set of zeroes and poles are ordered in an alternating form. In this way, the set of the real zeroes and real poles must have either of the following two forms:
$${\mathcal F}_{1}=\{a_{1}<b_{1}<a_{2}< \cdots <a_{L+1}<b_{L+1}\} \; \mbox{or} \; 
{\mathcal F}_{2}=\{b_{1}<a_{1}<b_{2}< \cdots <b_{L+1}<a_{L+1}\},$$
for a suitable $L \geq 0$. If we are in case ${\mathcal F}_{2}$, then we may perform a real movement to assume $b_{1}=0$ and $a_{1}=1$. If we now conjugate the real rational map by the real M\"obius transformation $T(z)=z/(2z-1)$, then we obtain that its corresponding set of zeroes and poles is of the required form.
\end{proof}

\begin{fact}\label{Fact2}
Let $R$ be a real rational map whose associated set ${\mathcal F}$ of its real zeroes and real poles is as in Fact \ref{Fact1}.
Then, we may perform real movements to connect it to any one of the form
\begin{equation}\label{eqrat2}
\left[\frac{\lambda (z-a_{1})^{n_{1}}(z-a_{2})^{n_{2}}(z-a_{3})^{n_{3}}\cdots(z-a_{L+1})^{n_{L+1}}}
{(z-b_{1})^{m_{1}}(z-b_{2})^{m_{2}}\cdots(z-b_{L})^{m_{L}}(z-b_{L+1})^{m_{L+1}}}\right] \in {\rm M}_{d}^{\mathbb R},
\end{equation}
where $d=n_{1}+\cdots+n_{L+1}=m_{1}+\cdots+m_{L+1}$, and $n_{j}, m_{j} \geq 1$. Moreover, if $d$ is odd, then we may assume $L$ to be even.
\end{fact}
\begin{proof} Let us consider the form \eqref{eqrat} for $R$, where $r=s=L+1$, and with ${\mathcal F}$ as in Fact \ref{Fact1}.

If some $l_{j}>1$ (respectively, $t_{j}>1$) and we write $l_{j}=2\widehat{l}_{j}+l'_{j}$ (respectively, $t_{j}=2\widehat{t}_{j}+t'_{j}$), where $\widehat{l}_{j}>0$ and $l'_{j} \geq 0$ (respectively, $\widehat{t}_{j}>0$ and $t'_{j} \geq 0$), then we may perform a real movement to transform the factor $(z-a_{j})^{l_{j}}$ (respectively, $(z-b_{j})^{t_{j}}$) into a factor of the form $(z-a_{j})^{l'_{j}}((z-(a_{j}+i\epsilon))(z-(a_{j}-i\epsilon)))^{\widehat{l}_{j}}$ (respectively, $(z-b_{j})^{t'_{j}}((z-(b_{j}+i\epsilon))(z-(b_{j}-i\epsilon)))^{\widehat{t}_{j}}$). 
Next, we may apply another real movement to collapse the complex conjugated zeroes (respectively, poles) to any other real zero (respectively, real pole). This procedure asserts the first assertion.

Now, assume that $d \geq 3$ and $L \geq 1$ are both odd. Then some $m_{j}$ must be even. So we may perform a real movement (as above with $m'_{j}=0$ and $\widehat{m}_{j}=m_{j}/2$) to move the pole $b_{j}$ into another different pole. Next, we have to collapse the real zeroes $a_{j-1}$ with $a_{j+1}$ (if $j=L+1$, then we collapse $a_{L+1}$ with $a_{1}$). After this process, we have decreased $L$ by one.
\end{proof}

All the above permits to assume that our real rational map $R$ is as described in Fact \ref{Fact2}.

\subsection*{I}
If $L=0$, then the above asserts that $[R]$ can be continuously connected (inside ${\rm M}_{d}^{\mathbb R}$) to $[\lambda (z-a_{1})^{d}/(z-b_{1})^{d}]$. In this case, we can perform a real movement that moves $a_{1}$ to $0$ and $b_{1}$ to $+\infty$, from which we connect $[R]$ to $[\lambda z^{d}]=[z^{d}]$.

\subsection*{II}
Assume $L\geq 2$ is even. In this case we can continuously move the points $a_{j}$ and $b_{j}$ (in the real line)  in order to assume that the reflection, whose circle of fixed points $\Sigma$ is orthogonal to the real line at the points $a_{1+L/2}$ and $b_{L+1}$, keeps invariant the set of all the zeroes and poles. In this case,  in Fact 2, we consider $n_{1+L/2}=d-L=m_{L+1}$ (all others $n_{j}$ and $m_{j}$ equal to $1$). We may now conjugate the rational map by a M\"obius transformation carrying the circle $\Sigma$  to the real line to obtain a (necessarily real) rational map with exactly one real zero and one real pole (all the others are non-real). Next, we can make a real movement to send all non-real zeroes (respectively, non-real poles) to the real zero (respectively, the real pole). Now, we are in the case $L=0$ as considered above in {\bf I}.

\subsection*{III}
Assume $L \geq 1$ is odd and $d \geq 2$ is even.
In this case we can continuously move the points $a_{j}$ and $b_{j}$ in order to assume that the reflection, whose circle $\Sigma$ of fixed points is orthogonal to the real line in the points 
$a_{1}$ and $a_{(L+3)/2}$, keeps invariant the set of zeroes and the set of poles. In this case, in Fact 2, we consider 
$n_{1}+n_{(L+3)/2}=d+1-L$ (and all others $n_{j}=1$) and $m_{j}=m_{L+2-j}$, for $j=1,\ldots,(L+1)/2$. 
 We may now conjugate the rational map by a M\"obius transformation carrying the circle $\Sigma$  to the real line to obtain a (necessarily real) rational map with exactly two real zeroes and all of its other zeroes and poles being non-real. By a real movement, we can collapse these two real zeroes and to send all non-real zeroes (respectively, non-real poles) to the real zero to a real pole. We are again in the case $L=0$ as in {\bf I}.

\medskip

Summarizing all the above, we have noted that every point $[R] \in {\rm M}_{d}^{\mathbb R}$ can be connected by a continuos arc (inside ${\rm M}_{d}^{\mathbb R}$) to the point $[z^{d}]$. This proves the lemma.
\end{proof}

\begin{rema}\label{nota0}
The ${\rm PSL}_{2}({\mathbb C})$-stabilizer of ${\rm Rat}_{d}({\mathbb R})$ is ${\rm PGL}_{2}({\mathbb R})$.
The (coarse) moduli space ${\rm M}_{d,{\mathbb R}}={\rm Rat}_{d}({\mathbb R})/{\rm PGL}_{2}({\mathbb R})$ is a real orbifold, defined over ${\mathbb Q}$, of dimension $2(d-1)$ \cite{Silv2}. There is a surjective map from ${\rm M}_{d,\mathbb R}$ onto ${\rm M}_{d}^{\mathbb R}$. It fails to be injective at $[\phi] \in {\rm M}_{d,\mathbb R}$ if and only if $\phi$ admits two different reflections as automorphisms which are not conjugated by its ${\rm PGL}_{2}({\mathbb R})$-stabilizer. In particular, this map is injective on a dense open subset.
\end{rema}

\subsection{Field of moduli / fields of definition}
\subsubsection{Fields of definition}
A field of definition of $\phi \in {\rm Rat}_{d}$ is a subfield ${\mathbb K}$ of ${\mathbb C}$ so that there is some rational map $\psi$, defined over ${\mathbb K}$, which is equivalent to $\phi$. 

\begin{lemm}\label{realista0}
Let $\phi \in {\rm Rat}_{d}$. Then $\phi$ is definable over ${\mathbb R}$ if and only if it admits a reflection as an automorphism.
\end{lemm}
\begin{proof}
If ${\mathbb R}$ is a field of definition of $\phi$, then there is some $\psi \in {\rm Rat}_{d}$, defined over the reals, and there is some $T \in {\rm PSL}_{2}({\mathbb C})$ such that $\phi=T \circ \psi \circ T^{-1}$. Since $J \circ \psi \circ J=\psi$, it follows that the reflection $T \circ J \circ T^{-1}$ is an automorphism of $\phi$. Conversely, if $\phi \in {\rm Rat}_{d}$ has a reflection $Q$ as an automorphism, then there is a $T \in {\rm PSL}_{2}({\mathbb C})$ such that $J=T \circ Q \circ T^{-1}$. Then $\psi=T \circ \phi \circ T^{-1}$ admits $J$ as an automorphism, so $\psi$ is defined over ${\mathbb R}$.
\end{proof}

\subsubsection{Field of moduli}
Let us denote by ${\rm Gal}({\mathbb C})$  the group of field automorphisms of ${\mathbb C}$.  The field of moduli of $\phi \in {\rm Rat}_{d}$ is the fixed field ${\mathcal M}_{\phi}$ of the group $\{\sigma \in {\rm Gal}({\mathbb C}):\phi \sim \phi^{\sigma}\}$. This is the field of definition of the point $[\phi]$  of the moduli space ${\rm M}_{d}$.
Following arguments due to Koizumi \cite{Koizumi}, it can be checked that ${\mathcal M}_{\phi}$ is the intersection of all the fields of definition of $\phi$. 

\begin{lemm}\label{realista2}
Let $\phi \in {\rm Rat}_{d}$. Then $\phi$ has field of moduli contained in ${\mathbb R}$ if and only if it admits an antiholomorphic automorphism.
\end{lemm}
\begin{proof}
The field of moduli of $\phi$ is contained in ${\mathbb R}$ if and only if $\phi \sim \overline{\phi}=J \circ \phi \circ J$. As previously noted (in the proof of Lemma \ref{realista}), this is equivalent for $\phi$ to has an antiholomorphic automorphism.
\end{proof}

In \cite{Silv1}, Silverman proved that rational maps, which are either of even degree or are equivalent to a polynomial, are definable over their field of moduli. In \cite{Hidalgo}, the first author noted that those rational maps, which cannot be defined over its field of moduli (so of odd degree), can be defined over a suitable quadratic extension of it.  By Lemma \ref{realista2}, real and pseudo-real rational maps are exactly those whose field of moduli is a subfield of ${\mathbb R}$, and (by Lemma \ref{realista0}) those having ${\mathbb R}$ as a field of definition are exactly the real ones. In particular, pseudo-real rational maps are examples of rational maps which cannot be defined over their field of moduli. 

\subsubsection{Silverman's cohomological obstructions}\label{cohomology}
In \cite{Silv1}, Silverman described an obstruction, in terms of cohomology, for a rational map $\phi$ to be definable over a subfield ${\mathbb K}$. In this section we recall this interpretation.
If ${\rm Aut}_{\mathbb K}({\mathbb C})$ is the group of field automorphisms of ${\mathbb C}$, acting as the identity on ${\mathbb K}$, then a {\it 1-cocycle} of $\phi$, with respect to ${\mathbb K}$, is a map 
$$\chi:{\rm Aut}_{\mathbb K}({\mathbb C}) \to {\rm PSL}_{2}({\mathbb C}): \sigma \mapsto T_{\sigma}$$
such that, $\forall \sigma, \tau \in {\rm Aut}_{\mathbb K}({\mathbb C})$ the following equalities hold
$$
\phi^{\sigma}=T_{\sigma} \circ \phi \circ T_{\sigma}^{-1}, \quad 
T_{\sigma \tau}=T_{\tau}^{\sigma} \circ T_{\sigma},
 $$
where $\phi^{\sigma}$ is the rational map obtained by applying $\sigma$ to the coefficients of $\phi$.  The set of such 1-cocycles, is the cohomology set $H^{1}\!\left({\rm Aut}_{\mathbb K}({\mathbb C}),{\rm PSL}_{2}({\mathbb C})\right)$. The existence of a 1-cocycle for $\phi$, with respect to ${\mathbb K}$, is not clear in general. 

By Weil's descent theorem \cite{Weil} (see also Theorem 16.37 in page 233 in \cite{Milne}), the existence of 
a 1-cocycle $\chi$ for $\phi$, with respect to ${\mathbb K}$, ensures the existence of a non-singular complex projective algebraic curve of genus zero $X$, defined over ${\mathbb K}$, and a biregular isomorphism $\varphi:\widehat{\mathbb C} \to X$, defined over $\overline{\mathbb K}$  (the algebraic closure of ${\mathbb K}$ inside ${\mathbb C}$), such that $\varphi=\varphi^{\sigma} \circ T_{\sigma}$, for every $\sigma \in  {\rm Aut}_{\mathbb K}({\mathbb C})$. It can be seen that  $\Psi=\varphi \circ \phi \circ \varphi^{-1}$ is also defined over ${\mathbb K}$. Now, if there is some isomorphism $L:X \to \widehat{\mathbb C}$, defined over ${\mathbb K}$, then $P=L \circ \varphi \in {\rm PSL}_{2}({\mathbb C})$ and the map $P \circ \phi \circ P^{-1}$ is defined over ${\mathbb K}$.  As a consequence of Riemann-Roch's theorem, in order to have the existence of such an $L$, we only need to check the existence of some ${\mathbb K}$-rational point in $X$. In terms of the 1-cocycle $\chi$, this is equivalent for it to be 
a {\it 1-coboundary}, that is, $\chi(\sigma)=(T^{\sigma})^{-1} \circ T$ for a suitable 
M\"obius transformation $T \in {\rm PSL}_{2}({\mathbb C})$. Summarizing the above, the rational map $\phi$ can be defined over the subfield ${\mathbb K}$ if it admits a 1-coboundary,  with respect to ${\mathbb K}$.

\section{Rational maps with non-trivial holomorphic/antiholomorphic automorphisms}\label{Sec:auto}
In this section, we provide a description of rational maps admitting antiholomorphic automorphisms and the dimensions of the corresponding spaces. We first recall the situation for holomorphic automorphisms as considered in \cite{MSW}.

\subsection{Rational maps with non-trivial holomorphic automorphisms}
Let $\phi \in {\rm Rat}_{d}$, where $d \geq 2$, admitting a non-trivial holomorphic automorphism $T_{n}$ of order $n \geq 2$. Up to conjugation by a suitable M\"obius transformation, we may assume  
$$T_{n}(z)=\omega_{n} z, \quad \omega_{n}=e^{2 \pi i/n},$$
in which case $\phi(z)=z\psi(z^{n})$, for a suitable rational map $\psi$.

\begin{theo}[\cite{MSW, HQ}]\label{teociclico}
Let $d,n \geq 2$ be integers. 
\begin{enumerate}
\item There is a rational map of degree $d$ admitting a holomorphic automorphism of order $n$ if and only if $d$ is congruent to either $-1,0,1$ modulo $n$. 
\item A rational map admitting a holomorphic automorphism of order $n$ is equivalent to one of the form $\phi(z)=z\psi(z^{n})$, where $$\psi(z)=\frac{\sum_{k=0}^{r} a_{k}z^{k}}{\sum_{k=0}^{r} b_{k}z^{k}} \in {\rm Rat}_{r},$$
satisfies that
\begin{itemize}
\item[(a)] $a_{r}b_{0}  \neq 0$,  if $d=nr+1$.

\item[(b)] $a_{r}\neq 0$ and $b_{0}=0$, if $d=nr$.

\item[(c)] $a_{r}=b_{0}=0$ and $b_{r} \neq 0$, if $d=nr-1$.
\end{itemize}
\end{enumerate}
\end{theo}

\begin{rema}
(i) If $d=2$, then  $(n,r) \in \{(2,1), (3,1)\}$; in other words, a rational map of degree two has holomorphic automorphisms of order either two or three.
(ii) In part (2) of Theorem \ref{teociclico},  
since $n \geq 2$, conditions {\rm (a)} and {\rm (b)} (respectively, {\rm (b)} and {\rm (c)}) are complementary. If $n>2$, then {\rm (a)} and {\rm (c)} are also complementary. 
\end{rema}

Let us  denote by ${\rm M}_{d,{\mathbb Z}_{n}}$ the locus in ${\rm M}_{d}$ consisting of the classes of rational maps admitting a holomorphic automorphism of order $n$. Then the branch locus ${\mathcal B}_{d}$ is the union of these ${\rm M}_{d,{\mathbb Z}_{n}}$, where $n \geq 2$, runs over those as in (1) in the previous result.

\begin{coro}[\cite{MSW}]\label{corodim}
Let $d, n \geq 2$ be integers so that $d$ is congruent to either $-1,0,1$ modulo $n$. Then ${\rm M}_{d,{\mathbb Z}_{n}}$ is connected and
$${\rm dim}_{\mathbb C}({\rm M}_{d,{\mathbb Z}_{n}})=\left\{ \begin{array}{ll}
2(d-1)/n, & d \equiv 1 \mod n.\\
(2d-n)/n, & d \equiv 0 \mod n.\\
2(d+1-n)/n, & d \equiv -1 \mod n.
\end{array}
\right.
$$
\end{coro}
\begin{proof}[Idea of the proof]
Let us consider the case $d \equiv 1 \mod n$, say $d=nr+1$ (the other two cases follow from similar arguments). By Theorem \ref{teociclico}, every rational map of degree $d$ admitting a holomorphic automorphism of order $n$ can be conjugated to one of the form $\phi(z)=z\psi(z^{n})$, where $\psi$ is as stated in part (2) of such theorem. If $\Omega \subset {\rm Rat}_{d}$ is the locus of such normalized rational maps, then it projects onto ${\rm M}_{d,{\mathbb Z}_{n}}$. There is a natural isomorphism between $\Omega$ and  the open and dense set 
${\rm Rat}_{r}={\mathbb P}^{2r+1}({\mathbb C}) \setminus {\rm Res}_{r}$, by sending the normalized rational map $\phi(z)=z\psi(z^{n})$ to $\psi(z)$.
Now, for the generic case (an open dense subset $\Omega_{0} \subset \Omega$), these rational maps have no extra automorphisms. So, if a M\"obius transformation $A \in {\rm PSL}_{2}({\mathbb C})$ conjugates one of these generic ones into another, then such a map must normalize the cyclic group $\langle T_{n}(z)=\omega_{n}z\rangle$, that is, it must be of the form 
$A_{\lambda}(z)=\lambda z$ or $B_{\lambda}(z)=\lambda/z$, where $\lambda \in {\mathbb C}^{*}:={\mathbb C}\setminus\{0\}$. If we set $G=\langle A_{\lambda}, B_{\lambda}: \lambda \in {\mathbb C}^{*}\rangle$, then it happens that $G$ is the normalizer of $\Omega$ in the group ${\rm PSL}_{2}({\mathbb C})$. As 
$\Omega_{0}/G$ is homeomorphic to an open and dense subset of ${\rm M}_{d,{\mathbb Z}_{n}}$ and it has dimension $2r=2(d-1)/n$, we are done in this case. 
\end{proof}

\begin{rema}
Corollary \ref{corodim}, for fixed $d \geq 2$, asserts that the dimension of ${\rm M}_{d,{\mathbb Z}_{n}}$ is maximal for $n=2$. In this case, 
${\rm dim}_{\mathbb C}({\rm M}_{d,{\mathbb Z}_{2}})=d-1$, in particular, ${\rm M}_{d,{\mathbb Z}_{2}} \neq \emptyset$. In \cite{HQ}, we used this observation to obtain the connectedness of ${\mathcal B}_{d}$.
\end{rema}

\begin{rema}
In \cite{MSW}, it can be found the general description for those rational maps admitting all possible finite subgroups of ${\rm PSL}_{2}({\mathbb C})$, together with the corresponding dimensions in moduli space.
\end{rema}

\subsection{Rational maps with antiholomorphic automorphisms}\label{Sec:antiauto}
Let $\phi \in {\rm Rat}_{d}$, $d \geq 2$, admitting an antiholomorphic automorphism $\tau$. As $\widehat{\rm Aut}(\phi)$ is a finite group,  $\tau$ must have finite order and, as it reverses the orientation, its order is even, say $2n$, for some $n \geq 1$. 
We may conjugate $\tau$ by a suitable M\"obius transformation (\cite[IE]{Maskit:book}) to either:
\begin{enumerate}
\item (if $n=1$) $J(z)=\overline{z}$ or $\tau_{1}(z)=-1/\overline{z}$; or
\item (if $n \geq 2$) to a power $\tau_{n}^{k}$, where $(k,n)=1$, and $$\tau_{n}(z)=\frac{\omega_{2n}}{\overline{z}}, \quad \omega_{2n}=e^{\pi i/n}.$$
\end{enumerate}
It follows that $\phi$ is equivalent to one admitting either $J$ or $\tau_{n}$ as an antiholomorphic automorphism.

\subsubsection{\bf The case of reflections}
If $n=1$ and $\tau$ is a reflection, then  $\phi$ is a real rational map.  In this case, $\phi$ is definable over ${\mathbb R}$. We already noted that ${\rm M}_{d}^{\mathbb R}$ is connected (Lemma \ref{Md(R)conexo}) and of real dimension $2(d-1)$.

\subsubsection{\bf The case of antiholomorphic automorphisms  different from reflections}\label{Sec:tn}
For each $n \geq 1$,  let us denote by $B_{d}(n) \subset {\rm Rat}_{d}$ the locus of rational maps of degree $d$ admitting $\tau_{n}$ as one of its antiholomorphic automorphisms. Let ${\mathcal B}_{d}(n) \subset {\rm M}_{d}$ be the locus of the equivalence classes  of the elements of $B_{d}(n)$. Note, in particular, that $B_{d}(1)$ consists of those rational maps admitting the imaginary reflection $\tau_{1}(z)=-1/\overline{z}$ as an automorphism.

Each rational map, admitting an antiholomorphic automorphism of order $2n$ (where, for $n=1$, we assume it is not a reflection), is equivalent to one in $B_{d}(n)$, in particular, ${\mathcal B}_{d}(n)$ consists of all equivalence classes of rational maps admitting an antiholomorphic automorphism of order $2n$ (which is not a reflection for $n=1$).

It may happen that some elements of $B_{d}(n)$ have a reflection as an antiholomorphic automorphism. We set 
$B_{d}^{\mathbb R}(n) \subset B_{d}(n)$ the sublocus of real rational maps inside $B_{d}(n)$ and by ${\mathcal B}^{\mathbb R}_{d}(n) = {\mathcal B}_{d}(n)  \cap {\rm M}^{\mathbb R}_{d}$ its corresponding  image sublocus. In particular, the set ${\mathcal P}_{d}(n)={\mathcal B}_{d}(n) \setminus {\mathcal B}_{d}^{\mathbb R}(n)$ (which it might be empty) is the locus in moduli space ${\rm M}_{d}$ of the equivalence classes of pseudo-real rational maps admitting an antiholomorphic automorphism of order $2n$.

\begin{rema}
If $n=2^{s}q$, where $q \geq 1$ is odd and $s \geq 0$, then ${\mathcal B}_{d}(n) \subset {\mathcal B}_{d}(2^{s})$. For instance, if 
$n$ is odd, then ${\mathcal B}_{d}(n) \subset {\mathcal B}_{d}(1)$.  
\end{rema}

\begin{prop}\label{interseccionreal}
If $ s \neq t$, then ${\mathcal B}_{d}(2^{s}) \cap {\mathcal B}_{d}(2^{t}) \subset {\rm M}_{d}^{\mathbb R}$.
\end{prop}
\begin{proof}
As $t \neq s$, we may assume $0 \leq s<t$, so we may write $t=s+l$, for some $l \geq 1$.
Let $[\phi] \in {\mathcal B}_{d}(2^{s}) \cap {\mathcal B}_{d}(2^{t})$, where 
$\phi \in B_{d}(2^{t})$. Then there is some $\phi_{1} \in B_{d}(2^{s})$ such that $[\phi]=[\phi_{1}]$. Then there is a M\"obius transformation $T \in {\rm PSL}_{2}({\mathbb C})$ such that $T \circ \phi_{1} \circ T^{-1}=\phi$, and  $\eta_{2}=T \circ \tau_{2^{s}} \circ T^{-1} \in \widehat{\rm Aut}(\phi)$ has order $2^{s+1}$.
If we set $\eta_{1}=\tau_{2^{t}}$, then 
$G=\langle \eta_{1},\eta_{2}\rangle \leq \widehat{\rm Aut}(\phi)$ and $G^{+}=G \cap {\rm Aut}(\phi)$. As $\eta_{1} \neq \eta_{2}$, the group $G^{+}$ is not the trivial group. If $G^{+}$ is not a cyclic group, then Lemma \ref{lema1} asserts that in $G$ there is a reflection, in particular, $[\phi] \in {\rm M}_{d}^{\mathbb R}$.
Let us now assume $G^{+}$ is a cyclic group, say generated by an elliptic transformation $E$. As $\eta_{1}^{2} \in G^{+}$ has order $2^{t}>1$, and its two fixed points are $0$ and $\infty$, it follows that the fixed points of $E$ are also $0$ and $\infty$. So, up to a suitable power, we may assume 
$E(z)=\omega_{n}z$, where $\omega_{n}=e^{2 \pi i/n}$, for $n \geq 2$. 
In this case, $\eta_{2}$ must keep invariant the set of fixed points of $E$, that is, the set $\{0,\infty\}$. If $\eta_{2}$ fixes $0$, then it must be a reflection and we are done. So, let us assume $\eta_{2}(0)=\infty$ and $\eta_{2}(\infty)=0$. This asserts that $\eta_{2}=\tau_{2^{s}}$ and, in particular, that $\phi \in B_{d}(2^{s}) \cap B_{d}(2^{t})$. As
$\eta_{1}(z)=e^{\pi i/2^{s+l}}/\overline{z},$
if we set $\alpha=\eta_{1}^{2}$, then 
$\alpha(z)=e^{\pi i/2^{s+l-1}}z$
is a holomorphic automorphism of $\phi$. As $\alpha^{L}(z)=e^{\pi i/2^{s}}z$, for $L=2^{l-1}$, one has $\alpha^{-L} \circ \tau_{2^{s}} (z)=1/\overline{z}$ is a reflection 
and an automorphism of $\phi$.
\end{proof}

\subsection*{The case of ${\bf n=1}$}
Let us first consider the case of rational maps admitting an imaginary reflection as an automorphism.

\begin{theo}\label{teoimaginary}
\hspace{2em}
\begin{enumerate}
\item If a rational map of degree $d$ admits an imaginary reflection as an antiholomorphic automorphism, then $d$ must be odd. 
\item Let $d$ be odd and let $\phi$ be a rational map of degree $d$. Then $\phi$ admits an imaginary reflection $\tau$ as an automorphism if and only if it can be conjugated to one of the form
$$\phi(z)=\dfrac{\sum_{k=0}^{d} a_{k} z^{k}}{\sum_{k=0}^{d} (-1)^{k} \overline{a_{d-k}} z^{k}} \in {\rm Rat}_{d},$$
where $\tau$ corresponds to $\tau_{1}(z)=-1/\overline{z}$.
\end{enumerate}
\end{theo}
\begin{proof}
Assume $\phi \in {\rm Rat}_{d}$ admits an imaginary reflection $\tau$ as an antiholomorphic automorphism. Up to conjugation by a suitable M\"obius transformation, we may assume that $\tau=\tau_{1}$. If 
$$\phi(z)=\dfrac{\sum_{k=0}^{d} \alpha_{k} z^{k}}{\sum_{k=0}^{d} \beta_{k} z^{k}} \in {\rm Rat}_{d},$$
the equality $\tau_{1} \circ \phi = \phi \circ \tau_{1}$ is equivalent to have the equality
$$\dfrac{\sum_{k=0}^{d} -\overline{\beta_{k}} z^{k}}{\sum_{k=0}^{d} \overline{\alpha_{k}} z^{k}}=
\dfrac{\sum_{k=0}^{d} (-1)^{d-k}\alpha_{d-k} z^{k}}{\sum_{k=0}^{d} (-1)^{d-k}\beta_{d-k} z^{k}},$$
which is also equivalent to the existence of some $\lambda \neq 0$ so that 
$$\overline{\beta_{k}}=\lambda (-1)^{d+1-k} \alpha_{d-k}, \;\; 
\overline{\alpha_{k}}=\lambda (-1)^{d-k} \beta_{d-k}.$$

The above, in particular, asserts that 
$1=|\lambda|^{2} (-1)^{d+1}$, so $d$ is odd, $|\lambda|=1$ and $\beta_{k}=\overline{\lambda} (-1)^{k} \overline{\alpha_{d-k}}$. If we set $\lambda=e^{2i\theta}$, then 
$e^{i\theta}\beta_{k}=(-1)^{k} \overline{e^{i\theta} \alpha_{d-k}}$. So if we take $a_{k}=e^{i\theta}\alpha_{k}$ and $b_{k}=e^{i\theta}\beta_{k}$, then we are done.
\end{proof}

\begin{rema}
The rational maps as described in part (2) of  Theorem \ref{teoimaginary} induce dianalytic maps on the real projective plane ${\mathbb P}^{2}({\mathbb R})$. 
The family of rational maps obtained in \cite[Proposition 2.5]{GH} is the subfamily as in the above theorem with $a_{1}=\cdots=a_{d-1}=0$. In the same paper, the authors prove that the bicritical pseudo-real rational maps  are of that form.
\end{rema}

Let ${\mathcal Z}_{d} \subset {\mathbb C}^{d+1}$ be the hypersurface defined by the points $\alpha:=(a_{0},\ldots,a_{d})$  such that 
 the resultant of the polynomials 
$P_{\alpha}(z)=\sum_{k=0}^{d} a_{k} z^{k}$ and $Q_{\alpha}(z)= \sum_{k=0}^{d} (-1)^{k} \overline{a_{d-k}} z^{k}$ is zero.
The quasi-projective variety ${\mathbb C}^{d+1} \setminus {\mathcal Z}_{d}$ is connected (as ${\mathcal Z}_{d}$ has real codimension two \cite{Duma}).

If $\alpha=(a_{0},\ldots,a_{d}) \in {\mathbb C}^{d+1} \setminus {\mathcal Z}_{d}$, then Theorem \ref{teoimaginary} asserts that $\phi_{\alpha} = P_{\alpha}/Q_{\alpha} \in B_{d}(1)$ and that every element of $B_{d}(1)$ is obtained in that way.

\begin{lemm}\label{lemacoro2}
Let $\alpha:=(a_{0},\ldots,a_{d}), \beta:=(b_{0},\ldots,b_{d}) \in {\mathbb C}^{d+1} \setminus {\mathcal Z}_{d}$. Then,  $\phi_{\beta}=\phi_{\alpha}$ if and only if  there is $\lambda \in {\mathbb R}^{*}:={\mathbb R}\setminus\{0\}$ such that $\beta=\lambda \alpha$.
\end{lemm}
\begin{proof}
Equality of the two rational maps is equivalent to the existence of some $\lambda \in {\mathbb C}\setminus\{0\}$ such that $b_{k}=\lambda a_{k}$ and $\overline{b_{k}}=\lambda \overline{a_{k}}$. This obligates to have $\lambda \in {\mathbb R}$.
\end{proof}

Let us denote by $\{a_{0}:\cdots:a_{d}\}$ the equivalence class of $(a_{0},\ldots,a_{d})$ in the quotient space $({\mathbb C}^{d+1} \setminus {\mathcal Z}_{d})/{\mathbb R}^{*}$. This is a connected real manifold, which can be seen as an open dense subset of ${\mathbb P}^{2d+1}({\mathbb R})$, so of real dimension $2d+1$.

All of the above is summarized in the following.

\begin{coro}\label{teoidentifica}
If $d$ is odd, then the map
$$\Xi: ({\mathbb C}^{d+1} \setminus {\mathcal Z}_{d})/{\mathbb R}^{*} \to B_{d}(1):
\{a_{0}:\cdots:a_{d}\} \mapsto \phi(z)=\dfrac{\sum_{k=0}^{d} a_{k} z^{k}}{\sum_{k=0}^{d} (-1)^{k} \overline{a_{d-k}} z^{k}}$$
provides a real analytic bijection. In particular, $B_{d}(1)$ and ${\mathcal B}_{d}(1)$ are connected and  $${\rm dim}_{\mathbb R}(B_{d}(1))=2d+1, \;\; 
{\rm dim}_{\mathbb R}({\mathcal B}_{d}(1))=2d-2.$$
\end{coro}
\begin{proof}
Part (2) of Theorem \ref{teoimaginary} asserts that the map $\Xi$ is well defined and surjective and Lemma \ref{lemacoro2} asserts the injectivity. 
Also, as defined, it can be seen to be real analytic. The real dimension of the space $({\mathbb C}^{d+1} \setminus {\mathcal Z}_{d})/{\mathbb R}^{*}$ is $2(d+1)-1=2d+1$, which is the required dimension for $B_{d}(1)$.

To obtain the dimension of ${\mathcal B}_{d}(1)$, we first observe that generically the elements of $B_{d}(1)$ has exactly one imaginary reflection as automorphism. Then, we only need to obtain the dimension of the ${\rm PSL}_{2}({\mathbb C})$-stabilizer of $\tau_{1}$ (i.e., the M\"obius transformations commuting with $\tau_1$). Direct computations show that these transformations are of the form $$A(z)=\frac{az+b}{-\overline{b}z+\overline{a}}, \; \mbox{where} \; |a|^{2}+|b|^{2}=1.$$
This permits to see that the real dimension of ${\mathcal B}_{d}(1)$ is equal to $(2d+1)-3=2d-2$, as required.
\end{proof}

Next, we proceed to describe $B_{d}^{\mathbb R}(1) \subset B_{d}(1)$. 
As for $d=1$ there are no pseudo-real maps,  we have that $B_{1}^{\mathbb R}(1) = B_{1}(1)$. We only need to take care of the case when $d \geq 3$ is odd.

\begin{theo}\label{conexon=1}
If $d \geq 3$ is odd, then
\begin{enumerate}
\item ${\rm dim}_{\mathbb R}(B_{d}^{\mathbb R}(1))=d+1$.
\item ${\rm dim}_{\mathbb R}({\mathcal B}_{d}^{\mathbb R}(1))=d-2$.
\item $B_{d}(1) \setminus B_{d}^{\mathbb R}(1)$ and ${\mathcal P}_{d}(1)={\mathcal B}_{d}(1) \setminus {\mathcal B}_{d}^{\mathbb R}(1)$ are both connected.
\end{enumerate}
\end{theo}
\begin{proof}
If $\phi \in B_{d}^{\mathbb R}(1)$, then there is a reflection $\rho$ as an antiholomorphic automorphism of it. Let us denote by $\Sigma_{\rho}$ the circle of fixed points of $\rho$.
The circle $\Sigma_{\rho}$ is either: (i) the unit circle $S^{1}$ or (ii) of the form $L \cup \{\infty\}$, where $L \subset {\mathbb C}$ is  an Euclidean line through $0$. In fact, 
if $\Sigma_{\rho}$ is not of the desired form, then the M\"obius transformation $\rho \circ \tau_{1}$ is a loxodromic transformation, which has infinite order, a contradiction.

Next, we proceed to analyze each of the two possibilities given above.

\subsection*{(I) ${\bf \Sigma_{\rho}=S^{1}}$}
In this case, $\rho(z)=1/\overline{z}$; so $T(z)=-z$ is a holomorphic automorphism of $\phi$ and, in particular,  $\phi(z)=z\psi(z^{2})$, for a suitable rational map $\psi \in {\rm Rat}_{r}$, where $r \in \{(d \pm 1)/2\}$. In this case, the map $\psi$ should satisfy the equality
$$(*) \quad \psi(z)=\frac{1}{\overline{\psi}(1/z)}.$$

As $\phi \in B_{d}(1)$, we may assume  
$$\phi(z)=\dfrac{\sum_{k=0}^{d} a_{k} z^{k}}{\sum_{k=0}^{d} (-1)^{k} \overline{a_{d-k}} z^{k}}.$$
Then $(*)$ above is equivalent to have either: (i) $a_{k}=0$ for $k$ odd or (ii) $a_{k}=0$ for $k$ even. In this way, these rational maps corresponds to the real $d$-dimensional sublocus 
${\mathcal A}$ of $W:=({\mathbb C}^{d+1} \setminus {\mathcal Z}_{d})/{\mathbb R}^{*}$ (as in Corollary \ref{teoidentifica}), where 
${\mathcal A}$ is the following set
$$ \{\{0:a_{1}:0:a_{3}:\cdots:0:a_{d}\} \in W\} \cup 
 \{\{a_{0}:0:a_{2}:\cdots:0:a_{d-1}:0\} \in W\}.$$

\subsection*{(II) $\Sigma_{\rho}$ is an Euclidean line through ${\bf 0}$ union ${\bf \infty}$}
In this situation, $\rho(z)=e^{2\alpha i} \overline{z}$. Again writing $\phi$ as above, the equality $\phi \circ \rho = \rho \circ \phi$ ensures that there is some $\mu \neq 0$ so that
$$\overline{a_{k}}=\mu a_{k} e^{2(k-1)\alpha i} \; \; \mbox{and} \;\;
a_{d-k}=\mu e^{2(\theta+k\alpha) i}\overline{a_{d-k}},$$
from where we obtain that $|\mu|=1$.
If $a_{0} \neq 0$, then the above (for $k=0,d$) asserts that 
$$\mu=e^{((1-d)\alpha-\theta)i} \;\; \mbox{ and } \;\;
a_{k}=|a_{k}| e^{((d+1-2k)\alpha+\theta)i/2}.$$

So the only free parameters are $r_{k}=|a_{k}|$, $k=0,...,d$. We obtain in this way a sublocus of real dimension $d+1$ (up to projectivization we may assume $r_{0}=1$) inside $B_{d}(1)$. The situation is similar for $a_{d} \neq 0$. 

All the above permits to obtain the desired dimensions as stated in (1) and (2) in our theorem. In order to obtain part (3), we only need to observe that 
as $B_{d}^{\mathbb R}(1)$ (respectively, ${\mathcal B}_{d}^{\mathbb R}(1)$) has codimension $d \geq 3$ in $B_{d}(1)$ (respectively, ${\mathcal B}_{d}(1)$), then the complement is connected (see, for instance, \cite[Corollary 0.3.11]{DV}).
\end{proof}

\begin{coro}\label{coro3}
If $d \geq 3$ is odd, then the locus in ${\rm M}_{d}$ of classes of pseudo-real rational maps which commute with an
imaginary reflection is connected and is of real dimension $2d-2$.
\end{coro}

\subsection*{The case of ${\bf n \geq 3}$ odd}
In this case, as $\tau_{n}^{n}=\tau_{1}$, we obtain that ${\mathcal B}_{d}(n) \subset {\mathcal B}_{d}(1)$ and 
${\mathcal B}_{d}^{\mathbb R}(n) \subset {\mathcal B}^{\mathbb R}_{d}(1)$.

\subsection*{The case of ${\bf n \geq 2}$ even}
In this case, $T(z)=\tau_{n}^{2}(z)=\omega_{n}z$. From Theorem \ref{teociclico}, $\phi(z)=z\psi(z^{n})$ for a suitable rational map$$\psi(z)=\dfrac{\sum_{k=0}^{r} a_{k} z^{k}}{\sum_{k=0}^{r} b_{k} z^{k}} \in {\rm Rat}_{r},$$ where $r \in \{(d-1)/n, d/n, (d+1)/n\}$ and 
\begin{enumerate}
\item $a_{r}b_{0} \neq 0$, if $r=(d-1)/n$;
\item $a_{r} \neq 0$ and $b_{0}=0$, if $r=d/n$;
\item $a_{r}=b_{0}=0$ and $b_{r} \neq 0$, if $r= (d+1)/n$.
\end{enumerate}

The equality $\tau_{n} \circ \phi = \phi \circ \tau_{n}$ is equivalent to have
$\psi(-1/z)=1/\overline{\psi}(z),$
which is equivalent to 
 the existence of some $\lambda \neq 0$ so that
$\overline{b_{k}}=\lambda (-1)^{r-k} a_{r-k},\;\;
\overline{a_{k}}=\lambda (-1)^{r-k} b_{r-k}.$
In particular, $|\lambda|=1$,  $r \geq 2$ is even and 
$b_{k}=\overline{\lambda} (-1)^{k} \overline{a_{r-k}}.$

The above asserts that case (2) is not possible.
Now, as $n \geq 2$ and $r \geq 2$ are even, the above also asserts the following.
\begin{enumerate}
\item If $r=(d-1)/n$, then $d \equiv 1 \mod(4)$ and $a_{r} \neq 0$.
\item if $r= (d+1)/n$, then $d \equiv 3 \mod(4)$ and $a_{r}=0$ and $a_{0}\neq 0$.
\end{enumerate}

\begin{rema}
As $n \geq 2$ does not divides $d$, only one of the values $(d-1)/n$ or $(d+1)/n$ can be even. For instance, if $d=3$, then $(d-1)/n=2/n$ cannot be even and $(d+1)/n=4/n$ is even only for $n=2$.
\end{rema}

Summarizing all the above is the following.

 \begin{theo}\label{teo9}
 Let $n,d \geq 2$ be integers with  $n$ even. 
\begin{enumerate}
\item  If $\phi \in {\rm Rat}_{d}$ admits an antiholomorphic automorphism $\tau$ of order $2n$, then $d \geq 3$ is odd and either of both,  $(d-1)/n$ or $(d+1)/n$, is an even integer, but not both of them. Moreover, if $r \in \{(d-1)/n, (d+1)/n\}$ is the even integer, then $\phi$ can be conjugated to a rational map of the form $\widehat{\phi}(z)=z\psi(z^{n})$, with $$\psi(z)=\dfrac{\sum_{k=0}^{r} a_{k} z^{k}}{\sum_{k=0}^{r} (-1)^{k}  \overline{a_{r-k}} z^{k}} \in {\rm Rat}_{r},$$ 
 where 
 \begin{itemize}
 \item[(i)] $a_{r} \neq 0$, for $r=(d-1)/n$; and
 \item[(ii)] $a_{r}=0$ and $a_{0} \neq 0$, for $r=(d+1)/n$.
 \end{itemize}
 In both cases, $\tau$ corresponds to $\tau_{n}$ (as previously defined in Section \ref{Sec:antiauto}).
\item Assume that $d \geq 3$ is odd and that either of both, $(d-1)/n$ or $(d+1)/n$, is an even integer, but not both of them. If  
$r \in \{(d-1)/n, (d+1)/n\}$ is the even integer, then the rational map $\widehat{\phi}(z)=z \psi(z^{n})$ (as in (1), together the provided conditions on the coefficients) admits the antiholomorphic automorphism $\tau_{n}$ of order $2n$.
  \end{enumerate}
\end{theo}

\begin{rema}\label{notita}
Let  $d \geq 3$ be odd and $n=2^{s}$, where $s \geq 1$ is such that one of the values $(d-1)/n$ or $(d+1)/n$ is an even integer but not both of them.
If we take all the coefficients $a_{k} \in {\mathbb R}$, then $\widehat{\phi}(z)=z \psi(z^{n})$ is real and admits an antiholomorphic automorphism of order $2n=2^{s+1}$, that is, ${\mathcal B}_{d}(2^{s}) \cap {\rm M}_{d}^{\mathbb R} \neq \emptyset$. Now, for the case $s=0$, the last fact holds as ${\mathcal B}_{d}^{\mathbb R}(1) \neq \emptyset$ (by Theorem \ref{conexon=1}).
\end{rema}

\begin{exam}[Antiholomorphic automorphisms in degree $d=3$]
Let $\phi \in {\rm Rat}_{3}$ admitting an antiholomorphic automorphism of order $2n$, where $n \geq 2$.
Theorem \ref{teo9} asserts that the only possibility is $n=2$ (case (ii) of (1) and $r=2$) and,
up to conjugation, we may assume that  
$$\phi(z)=\frac{a+bz^{2}}{-\overline{b}z+\overline{a} z^{3}}; \quad a \neq 0.$$

In this case, the antiholomorphic automorphism is given by $\tau_{2}(z)=i / \overline{z}$. As $a \neq 0$, we may assume $a=1$. 
It follows that $\phi$ is equivalent to a rational map of the form
$$\widehat{\phi}(z)=\frac{1+c z^{2}}{-\overline{c}z+z^{3}}.$$

In this way, a rational map of degree $d=3$ only may admits antiholomorphic automorphisms of order $2$ or $4$. By Corollary \ref{teoidentifica}, the locus
in ${\rm M}_{3}$ of (classes of) rational maps admitting antiholomorphic automorphisms of order $2$ is connected of real dimension $4$ and, by the previous computations, those admitting one of order $4$ is connected of real dimension $2$.
\end{exam}

We may proceed in a similar fashion as done for Corollary \ref{teoidentifica} to obtain 
the real dimension of the locus  ${\mathcal B}_{d}(n)$, for $n \geq 2$ even.

As before, ${\mathcal Z}_{r} \subset {\mathbb C}^{r+1}$ is the hypersurface defined by the points $\alpha:=(a_{0},\ldots,a_{r})$  such that 
 the resultant of the polynomials 
$P_{\alpha}(z)=\sum_{k=0}^{r} a_{k} z^{k}$ and $Q_{\alpha}(z)= \sum_{k=0}^{r} (-1)^{k} \overline{a_{d-k}} z^{k}$ is zero, and 
the quasi-projective variety ${\mathbb C}^{r+1} \setminus {\mathcal Z}_{r}$ is a connected real manifold.

If $\alpha=(a_{0},\ldots,a_{r}) \in {\mathbb C}^{r+1} \setminus {\mathcal Z}_{d}$, then Theorem \ref{teo9} asserts that $\phi_{\alpha} = P_{\alpha}/Q_{\alpha} \in B_{d}(n)$ and that every element of $B_{d}(n)$ is obtained in that way. Moreover, by Lemma  \ref{lemacoro2}, 
for $\alpha=(a_{0},\ldots,a_{r}), \beta=(b_{0},\ldots,b_{r}) \in {\mathbb C}^{r+1} \setminus {\mathcal Z}_{d}$, one has that  $\phi_{\beta}=\phi_{\alpha}$ if and only if  there is $\lambda \in {\mathbb R}^{*}:={\mathbb R}\setminus\{0\}$ such that $\beta=\lambda \alpha$.
Let us denote by $\{a_{0}:\cdots:a_{r}\}$ the equivalence class of $(a_{0},\ldots,a_{r})$ in the connected real manifold $({\mathbb C}^{r+1} \setminus {\mathcal Z}_{d})/{\mathbb R}^{*}$of real dimension $2r+1$.

\begin{coro}\label{conexidad}
Let $n \geq 2$ be even and $d \geq 3$ odd so that either $(d-1)/n$ or $(d+1)/n$ is even.
\begin{enumerate}
\item If $r=(d-1)/n$ is even, then the map 
$$({\mathbb C}^{r+1} \setminus {\mathcal Z}_{d})/{\mathbb R}^{*} \to B_{d}(n):
\{a_{0}:\cdots:a_{r}\} \mapsto \widehat{\phi}(z)= z \psi(z^{n})=z \dfrac{\sum_{k=0}^{r} a_{k} z^{kn}}{\sum_{k=0}^{r} (-1)^{k} \overline{a_{r-k}} z^{kn}}$$
provides a real analytic bijection.
In particular, $B_{d}(n)$ and ${\mathcal B}_{d}(n)$ are both connected and
$$
{\rm dim}_{\mathbb R}(B_{d}(n))=2r+1, \;\; 
{\rm dim}_{\mathbb R}({\mathcal B}_{d}(n))=2r.
$$ 

\item If $r=(d+1)/n$ is even, then the map 
$$({\mathbb C}^{r} \setminus {\mathcal Z}_{d})/{\mathbb R}^{*} \to B_{d}(n):
\{a_{0}:\cdots:a_{r-1}\} \mapsto \widehat{\phi}(z)= z \psi(z^{n})= z \dfrac{\sum_{k=0}^{r-1} a_{k} z^{kn}}{\sum_{k=1}^{r} (-1)^{k} \overline{a_{r-k}} z^{kn}}$$
provides a real analytic bijection.
In particular, $B_{d}(n)$ and ${\mathcal B}_{d}(n)$ are both connected and $${\rm dim}_{\mathbb R}(B_{d}(n))=2r-1, \;\; {\rm dim}_{\mathbb R}({\mathcal B}_{d}(n))=2r-2.$$ 
\end{enumerate}
\end{coro}

\begin{coro}\label{coron=2}
If $d \geq 3$ is odd, then there are pseudo-real rational maps of degree $d$ admitting an antiholomorphic automorphism of order $4$. 
\end{coro}
\begin{proof}
Write $d=2l+1$, where $l \geq 1$. If $l$ is odd, then we take $r=(d+1)/2$ and if $l$ is even, then we take $r=(d-1)/2$ in Theorem \ref{teo9} (where we are taking $n=2$). By the previous Corollary, ${\mathcal B}_{d}(2)$ has positive real dimension.  Now, we only need to observe that the real part ${\mathcal B}_{d}^{\mathbb R}(2)$ has dimension at most $r-1$.
\end{proof}

\begin{coro}\label{realconexo}
The locus ${\rm M}_{d}({\mathbb R})$ of real points  in moduli space ${\rm M}_{d}$ is connected.
\end{coro}
\begin{proof}
If $d \geq 2$ is even, then ${\rm M}_{d}({\mathbb R})={\rm M}_{d}^{\mathbb R}$, so it is connected (Lemma \ref{Md(R)conexo}).
In the case that $d \geq 3$ is odd, 
${\rm M}_{d}({\mathbb R})$ is the union of the loci ${\rm M}_{d}^{\mathbb R}$ and others of the form ${\mathcal B}_{d}(2^{s})$. Since (i) each of them is connected and (ii) ${\rm M}_{d}^{\mathbb R} \cap {\mathcal B}_{d}(2^{s}) \neq \emptyset$ (Remark \ref{notita}), we obtain the connectedness.
\end{proof}

\begin{rema}
As a consequence of Theorem \ref{teo9}, there is no rational map of degree $d \in \{3,5\}$ admitting an antiholomorphic automorphism of order $8$ and, for $d=7$, there is a family of real dimension $2$ with such a property.
\end{rema}

\section{Pseudo-real rational maps}\label{Sec:main}
\subsection{Automorphisms of pseudo-real rational maps}
We start with a simple observation about  the group of holomorphic automorphisms of  a pseudo-real rational map.

The finite groups of M\"obius transformations are either: the trivial group, cyclic groups, dihedral groups or one of the Platonic groups ${\mathcal A}_{4}$, ${\mathcal A}_{5}$ or ${\mathfrak S}_{4}$ \cite{Beardon}. The quotient orbifold $\widehat{\mathbb C}/G$, where $G$ is any of these finite groups, can be identified with $\widehat{\mathbb C}$ and it has at most three   cone points (exactly three when $G$ is neither trivial or cyclic).

\begin{theo}\label{autopseudo}
If $\phi \in {\rm Rat}_{d}$ is a pseudo-real rational map, then  ${\rm Aut}(\phi)$ is either trivial or a cyclic group. 
\end{theo}
\begin{proof}
As $\phi \in {\rm Rat}_{d}$ is a pseudo-real rational map, $d \geq 3$ is odd. In particular, $G^{+}={\rm Aut}(\phi)<{\rm PSL}_{2}({\mathbb C})$ is a finite group of M\"obius transformations.
Let us assume $G^{+}$ is neither the trivial one or a cyclic group, so the quotient orbifold $\widehat{\mathbb C}/G^{+}$ can be identified with $\widehat{\mathbb C}$ and the three cone points with $\infty, 0$ and $1$. Let us consider a regular branched covering $\pi:\widehat{\mathbb C} \to \widehat{\mathbb C}$ with $G^{+}$ as its deck group. Let us denote by $\pi_{0}:\widehat{\mathbb C} \setminus \pi^{-1}(\{0,1,\infty\}) \to \widehat{\mathbb C}\setminus \{0,1,\infty\}$ the restriction of $\pi$, which turns out to be a regular covering map with deck group $G^{+}$. As $G=\widehat{\rm Aut}(\phi)$ normalizes $G^{+}$, there is an antiholomorphic automorphism $\widehat{Q}$, of order two,  such that  $\pi \circ Q=\widehat{Q} \circ \pi$, for every $Q \in G \setminus G^{+}$.
 As $\widehat{Q}$ must keep invariant the branch locus $\{0,1,\infty\}$ and it has order two, it must have a fixed point; so it is a reflection. This implies that $\widehat{Q}$ has infinitely many fixed points different from the branch values of $\pi$. Let $z \in {\mathbb C}\setminus\{0,1\}$ be a fixed point of $\widehat{Q}$ and let $w \in \widehat{\mathbb C}$ such that $\pi(w)=z$. If $Q \in G\setminus G^{+}$, then $Q(w)$ and $w$ are in the same $G^{+}$-orbit. It follows that there is some $T \in G^{+}$ such that 
 $Q \circ T$ fixes $w$. The only possibility of an extended M\"obius transformation (of finite order) to have fixed points is for it to be a reflection, a contradiction (pseudo-real rational maps have no reflections as automorphisms).
\end{proof}

\subsection{A characterization of pseudo-real rational maps}
Next, a characterization of a pseudo-real rational map, in terms of the different possibilities for its group of holomorphic automorphisms, is provided.

\subsubsection{\bf Trivial group of automorphisms}
If there is no non-trivial holomorphic automorphism, then
 the pseudo-real property is easy to state as already observed by Silverman \cite{Silv1,Silv4} (for completeness, a proof is provided).

\begin{theo}[\cite{Silv1}]\label{autotrivial}
Let $\phi \in {\rm Rat}_{d}$, $d \geq 3$ odd, such that ${\rm Aut}(\phi)$ is trivial. Then $\phi$ is pseudo-real if and only if there is an imaginary reflection that is an automorphism of $\phi$.
\end{theo}
\begin{proof}
If $\phi$ is pseudo-real, then necessarily there is some $Q \in \widehat{\rm Aut}(\phi)$ that is antiholomorphic. Since ${\rm Aut}(\phi)$ is trivial, then $Q^{2}=I$. Since $\phi$ is not real, $Q$ cannot be a reflection; so it is an imaginary reflection. In the other direction, if there is some imaginary reflection $Q \in \widehat{\rm Aut}(\phi)$, then there is no reflection $R \in \widehat{\rm Aut}(\phi)$ as, in such a case, $R \circ Q \in {\rm Aut}(\phi)$ is different from the identity.
\end{proof}

As every imaginary reflection can be conjugated, by a suitable M\"obius transformation, to $\tau(z)=-1/\overline{z}$, the previous result can also be stated as follows (see also \cite[section 2]{Filom}).

\begin{coro}
Let $\phi \in {\rm Rat}_{d}$, $d \geq 3$ odd, so that ${\rm Aut}(\phi)$ is trivial. Then $\phi$ is pseudo-real if and only if there is a M\"obius transformation $T \in {\rm PSL}_{2}({\mathbb C})$ so that $\eta=T \circ \phi \circ T^{-1}$ satisfies that
$\eta(-1/\overline{z})=-1/\overline{\eta(z)}$.
\end{coro}

\subsubsection{\bf Non-trivial cyclic group of automorphisms}
If ${\rm Aut}(\phi)=\langle T \rangle \cong {\mathbb Z}_{n}$, where $n \geq 2$, then we may assume, up to conjugation by a suitable M\"obius transformation, that $\phi(z)=z\psi(z^{n})$, for a suitable rational map $\psi$, and $T(z)=\omega_{n} z$, where $\omega_{n}=e^{2 \pi i/n}$ \cite{MSW} (see also Section \ref{Sec:auto}). Using this normalization, the following description is provided.

\begin{theo}\label{teo7}
Let $\phi \in {\rm Rat}_{d}$, $d \geq 3$ odd, with ${\rm Aut}(\phi) \cong {\mathbb Z}_{n}$ for $n \geq 2$. Then $\phi$ is pseudo-real if and only if 
it can be conjugated to a rational map $\widetilde{\phi}(z)=z \psi(z^{n})$ so that 
\begin{enumerate}
\item $\psi(z) \neq \overline{\psi}(e^{i\theta}z)$, for every $\theta \in {\mathbb R}$, and

\item  there is some $\alpha \in {\mathbb C}\setminus\{0\}$ with  $|\alpha|=1$ and $\alpha^{n} \neq 1$, so that
$\psi(z)\overline{\psi}(\alpha^{n}/z)=1$.
\end{enumerate}
\end{theo}
\begin{proof}
Let us now consider a rational map $\phi \in {\rm Rat}_{d}$, $d \geq 3$ odd, with ${\rm Aut}(\phi) \cong {\mathbb Z}_{n}$, $n \geq 2$. Up to equivalence, we may assume that $\phi(z)=z\psi(z^{n})$, where $\psi$ is some rational map of degree $r \geq 1$,
${\rm Aut}(\phi)=\langle T(z)=\omega_{n} z\rangle \cong {\mathbb Z}_{n}$, where $\omega_{n}=e^{2 \pi i/n}$. Let us recall that the rational map $\phi$ is pseudo-real if and only if it admits an antiholomorphic automorphism and it has no reflections as automorphisms. Let $Q$ be an antiholomorphic automorphism of $\phi$. As $Q^{2} \in {\rm Aut}(\phi)$, we must have that $Q$ keeps invariant the set $\{\infty,0\}$. In this way, we must have one of  two possibilities: (1) $Q(0)=0$ and $Q(\infty)=\infty$ or (2) $Q(0)=\infty$  and $Q(\infty)=0$.

In case (1), $Q(z)=\beta \overline{z}$, where $|\beta|=1$, is a reflection, and in this case, 
$\phi$ is real, and $\psi(z)=\overline{\psi}(z/\beta^{n})$. In order to rule out this possibility, the rational map $\psi$ must satisfies that $\psi(z) \neq \overline{\psi}(e^{i\theta}z)$, for every $\theta \in {\mathbb R}$.

In case (2), $Q(z)=\alpha/\overline{z}$, where $\alpha/\overline{\alpha}=\omega_{n}^{k}$ for some $k=0,1,...,n-1$. This is equivalent for $\psi$ to satisfy the equality $\psi(z)\overline{\psi}(\alpha^{n}/z)=1$. We may conjugate $\phi$ by $T(z)=\lambda z$, where $|\lambda|^{2}=1/|\alpha|$ in order to assume that $\alpha=e^{i\varphi}$. As all elements of $\widehat{\rm Aut}(\phi) \setminus {\rm Aut}(\phi)$ are of the form
$L(z)=\frac{\alpha \omega_{n}^{s}}{\overline{z}}$,
where $s \in \{0,1,...,n-1\}$, if $\varphi \neq 2s\pi/n $, for every 
$s=0,1,...,n-1$, then $\phi$ will be pseudo-real. 
\end{proof}

\subsection{Examples}
In the above, we have provided necessary and sufficient conditions for a rational map, with a cyclic group of automorphisms, to be pseudo-real. In this section, we construct explicit examples.

\subsubsection{\bf Non-trivial group of holomorphic automorphisms}\label{ejemplos}
Let $n \geq 4$ be an integer and let $\psi(z)$ be a rational map satisfying the following three properties.

\begin{enumerate}
\item[(1)] There is no rational map $\rho$ so that $\psi(z)=\rho(z^{m})$, for some $m \geq 2$.

\item[(2)] For every $s \in {\mathbb C}\setminus\{0\}$,  $\psi(z) \neq \dfrac{1}{\psi(s/z)}$.

\item[(3)] $\overline{\psi}(z)=\dfrac{1}{\psi(-1/z)}$.

\end{enumerate}

Set $\phi(z)=z\psi(z^{n})$ and note that $T(z)=\omega_{n} z \in {\rm Aut}(\phi)$.
If $n=4$, then (as the groups ${\mathcal A}_{4}$ and ${\mathcal A}_{5}$ do not contain an element of order $4$), we have that 
 ${\rm Aut}(\phi)$ is either a cyclic group or a dihedral group or ${\mathfrak S}_{4}$. Similarly, if 
$n \geq 6$, then ${\rm Aut}(\phi)$ is either a cyclic group or a dihedral group. 
Condition (2) ensures that ${\rm Aut}(\phi)$ cannot be dihedral (so neither can be ${\mathfrak S}_{4}$ for $n=4$).
Condition (1) ensures that ${\rm Aut}(\phi)=\langle T \rangle$.
Condition (3) ensures that $\tau_{n}(z)=\omega_{2n}/\overline{z}$ is an antiholomorphic automorphism of $\phi$. Now, every 
antiholomorphic automorphism of $\phi$ must be of the form $$T^{l} \circ \tau_{n}(z)=\omega_{n}^{l}\omega_{2n}/\overline{z}, \quad l=0,1,...,n-1,$$
and, since $\omega_{n}^{l}\omega_{2n} \neq 1$, none is a reflection. As a consequence, $\phi$ is a pseudo-real rational map with ${\rm Aut}(\phi) \cong {\mathbb Z}_{n}$.

Now, the above asserts that, in order to construct explicit examples of pseudo-real rational maps $\phi$ with ${\rm Aut}(\phi) \cong {\mathbb Z}_{n}$ ($n \geq 6$), we only need to find explicit rational maps $\psi$ satisfying (1), (2) and (3) above. Let us consider 
$$\psi(z)=\frac{\sum_{k=0}^{r} \alpha_{k} z^{k}}{\sum_{k=0}^{r} \beta_{k} z^{k}}, \quad \alpha_{r} \neq 0.$$

In this case, $\phi$ will have degree $d=1+nr$. Condition (3) is equivalent to: (i) $r$ is even (so $d$ will be odd as supposed to be) and (ii) there  exists some $\theta \in {\mathbb R}$ so that, for every $k=0,1,...,r$, 
$\beta_{k}=(-1)^{k}e^{i\theta}\overline{\alpha_{r-k}}$ holds. To satisfy Condition (2) we only need to assume (by evaluation at $z=0$)
$$\alpha_{0}\alpha_{r} \neq e^{2 i \theta} \overline{\alpha_{0}\alpha_{r}},$$
and Condition (1) will be satisfied if we also assume $\alpha_{1} \neq 0$. 

By setting $a_{k}=e^{-i\theta/2}\alpha_{k}$ and $b_{k}=e^{-i\theta}\beta_{k}$, all the above can be summarized in the following.

\begin{theo}\label{T4}
Let $n \geq 4$ and let $r \geq 2$ be even. If 
$a_{0},..., a_{r} \in {\mathbb C}$ are such that $a_{1}a_{r} \neq 0$ and $a_{0}a_{r} \notin {\mathbb R}$, and
$$\psi(z)=\frac{\sum_{k=0}^{r} a_{k} z^{k}}{\sum_{k=0}^{r} (-1)^{k}\overline{a_{r-k}} z^{k}},$$
then $\phi(z)=z\psi(z^{n})$ is a pseudo-real rational map with ${\rm Aut}(\phi)=\langle T(z)=\omega_{n} z\rangle \cong {\mathbb Z}_{n}$.
\end{theo}

\begin{exam}
If, in Theorem \ref{T4}, we have $n=6$, $r=2$,  $a_{0}=1$, $a_{1}=1$, $a_{2}=i$, then 
$$\psi(z)=\frac{1+z + iz^{2}}{-i -z+z^{2}}, \;\; \mbox{ and } \;\;
\phi(z)=z\frac{1+z^6 + iz^{12}}{-i -z^6+z^{12}}$$
is a pseudo-real rational map of degree $13$ with ${\rm Aut}(\phi)=\langle T(z)=e^{\pi i/3}z \rangle\cong {\mathbb Z}_{6}$.
\end{exam}

\subsubsection{\bf Trivial group of holomorphic automorphisms}
As previously noted, in \cite{Silv1}, Silverman constructed pseudo-real rational maps in every odd degree $d \geq 3$ with trivial group of holomorphic automorphisms. Theorem \ref{teo7} permits to provide more examples of such type as quotients of pseudo-real rational maps with non-trivial groups of holomorphic automorphisms.

\begin{theo} \label{teo10}
Let $n \geq 2$ and $\psi$ be a rational map satisfying the following properties:
\begin{enumerate}
\item $\psi(z)^{n} \neq \overline{\psi}(e^{i\theta}z)^{n}$, for every $\theta \in {\mathbb R}$, and
\item there exists $\alpha \in {\mathbb C} \setminus\{0\}$, $|\alpha|=1$, $\alpha^{n} \neq 1$, 
such that $\psi(z)\overline{\psi}(\alpha^{n}/z)=1$.
\item for every $m \geq 2$ which is a divisor of $n$, $(\psi(\omega_{m}z)/\psi(z))^{n}\neq 1$.
\end{enumerate}

Then the rational map  $\widehat{\phi}(w)=w \psi(w)^{n}$ is pseudo-real with a trivial group of holomorphic automorphisms. 
\end{theo}
\begin{proof}
If we set  $\phi(z)=z \psi(z^{n})$, then hypothesis (1) and (2) asserts (by Theorem \ref{teo7}) that $\phi$ is a pseudo-real rational map with $\widehat{\rm Aut}(\phi)=\langle \tau_{n}(z)=\omega_{2n}/\overline{z} \rangle \cong {\mathbb Z}_{2n}$ and ${\rm Aut}(\phi)=\langle T_{n}(z)=\omega_{n}z\rangle \cong {\mathbb Z}_{n}$. 
Let us consider the regular branch cover $w=\pi(z)=z^{n}$, whose deck group is the cyclic group ${\rm Aut}(\phi)$. Then 
(i) $\pi \circ \phi = \widehat{\phi} \circ \pi$ and (ii) $\widehat{\phi}$ admits the imaginary reflection $\tau_{1}(w)=-1/\overline{w}$ as antiholomorphic automorphism. If ${\rm Aut}(\widehat{\phi})$ is the trivial group, then $\widehat{\phi}$ is a pseudo-real rational map with a trivial group of holomorphic automorphisms as desired. 

Let us assume that ${\rm Aut}(\widehat{\phi})$ is non-trivial and let $V \in {\rm Aut}(\widehat{\phi})$ be a non-trivial holomorphic automorphism of maximal order.
As $\phi$ sends $\{0,\infty\}$ into itself, $\pi(0)=0$, $\pi(\infty)=\infty$ and $\pi \circ \phi = \widehat{\phi} \circ \pi$, the map $\widehat{\phi}$ does the same. In particular, 
$V$ must keep the set $\{0,\infty\}$ invariant. So either (i) $V(w)=\beta/w$, where $\beta \neq 0$, or (ii) $V(w)=\beta w$, where $\beta=e^{2k \pi i/m}$ for some $m \geq 2$ and $(k,m)=1$ (without lost of generality we may assume $k=1$ as $\langle V \rangle \leq {\rm Aut}(\widehat{\phi})$).

If $V(w)=\beta/w$, then $\eta=\tau_{1} \circ V \in \widehat{\rm Aut}(\widehat{\phi})$. As $\eta(w)=-\overline{\beta}^{-1} \overline{w}$, it holds that  
$W=\eta^{2} \in {\rm Aut}(\widehat{\phi})$ has the form 
$W(w)=|\beta|^{-2} w$. As $W$ must be of finite order, necessarily $|\beta|=1$, that is, $\beta=-e^{i\theta}$ for a suitable $\theta \in {\mathbb R}$. In this way, 
$\widehat{\phi}$ admits the reflection $\eta(w)=e^{i\theta}\overline{w}$ as antiholomorphic automorphism. The condition $\eta \circ \widehat{\phi} \circ \eta=\widehat{\phi}$ is equivalent to have $\overline{\psi}(e^{-\theta i}w)^{n}=\psi(w)^{n}$, a contradiction to our hypothesis (1).

If  $V(w)=\omega_{m}w$, where $m \geq 2$, then the condition $\widehat{\phi}=V^{-1} \circ \widehat{\phi} \circ V$ asserts that 
$\psi(w)^{n}=\psi(\omega_{m} w)^{n}$, for every $w \in \widehat{\mathbb C}$. In particular, $\psi(\omega_{m}w)=\omega_{n}^{l} \psi(w)$, for a suitable $0 \leq l \leq m-1$. In particular, if we set $L(z)=\omega_{m} z$, then 
$$L^{-1} \circ \phi \circ L(z)=z \psi(\omega_{m}^{n} z^{n})= z \omega_{n}^{ln}\psi(z^{n})=z\psi(z^{n})=\phi(z),$$
that is, $L \in {\rm Aut}(\phi)=\langle T_{n} \rangle$. Then $m \geq 2$ must be a divisor of $n$.
The above produces a contradiction to our hypothesis (3).
\end{proof}

\subsection{On the connectedness of pseudo-real locus in moduli space}\label{Sec:teo6}
Let us denote by ${\mathcal P}_{d}$ the sublocus in ${\rm M}_{d}$ consisting of the equivalence classes of pseudo-real rational maps. We already know that, for $d$ even, ${\mathcal P}_{d}=\emptyset$  and, for every $d \geq 3$ odd, ${\mathcal P}_{d} \neq \emptyset$ \cite{Silv1}.

\begin{theo}\label{disconexo}
If $d \geq 3$ is odd, then ${\mathcal P}_{d}$ is not connected.
\end{theo}
\begin{proof}
For each $n \geq 1$, ${\mathcal P}_{d}(n)={\mathcal B}_{d}(n) \setminus {\mathcal B}_{d}^{\mathbb R}(n)$ is the sublocus in ${\rm M}_{d}$ of those classes of pseudo-real  rational maps admitting an antiholomorphic automorphism of order $2n$. By Theorem \ref{teo9}, we know that $n$ does not divides $d$. Corollary \ref{coro3} asserts that ${\mathcal P}_{d}(1) \neq \emptyset$
and Theorem \ref{conexon=1} asserts that it is connected.
Clearly, $$\bigcup_{n \geq 1 \; odd} {\mathcal P}_{d}(n) = {\mathcal P}_{d}(1).$$

Corollary \ref{coron=2} asserts the existence of pseudo-real rational maps admitting an antiholomorphic automorphism of order $4$, that is, ${\mathcal P}_{d}(2) \neq \emptyset$.  As a consequence of Proposition \ref{interseccionreal},  for each $n \geq 2$ even,  ${\mathcal P}_{d}(1) \cap {\mathcal P}_{d}(n) =\emptyset$.
\end{proof}

\begin{rema}[On the number of connected components of ${\mathcal P}_{d}$]
Theorem \ref{disconexo} asserts that, for $d \geq 3$ odd, the locus ${\mathcal P}_{d}$ is disconnected. We may wonder in the number of its connected components. We first note that, 
if $n=2^{s}q$, where $q \geq 1$ is odd,  then ${\mathcal P}_{d}(n) \subset {\mathcal P}_{d}(2^{s})$. 
So, if $N(d)=\{s \geq 0: {\mathcal B}_{d}(2^{s}) \neq \emptyset\}$, then ${\mathcal P}_{d}=\cup_{s \in N(d)} {\mathcal P}_{d}(2^{s})$. As a consequence of Proposition \ref{interseccionreal}, for $s \neq t$, ${\mathcal P}_{d}(2^{t}) \cap {\mathcal P}_{d}(2^{s}) = \emptyset$.
For each $s \in N(d)$,  ${\mathcal P}_{d}(2^{s}) \neq \emptyset$. In fact, 
the case $s=0$ is given by Corollary \ref{coro3}, the case, $s=1$ by Corollary \ref{coron=2} and 
Theorem \ref{T4} for $s \geq 2$. It follows that the number of connected components of ${\mathcal P}_{d}$ is at least the cardinality of $N(d)$ (the equality holds if each ${\mathcal P}_{2}(2^{s})$ is connected).
Now, if $d-1=2^{\alpha_{d}}q$ and $d+1=2^{\beta_{d}}p$, where $\alpha_{d}, \beta_{d} \geq 1$ and $q,p$ are odd, then (as a consequence of Corollary \ref{teoidentifica} and Corollary \ref{conexidad})  $|N(d)|=1+{\rm Max}\{\alpha_{d},\beta_{d}\}$. 
\end{rema}

\section{Real rational maps}\label{Sec:real}

\subsection{Examples of real rational maps not definable over their field of moduli}
In  \cite{Silv1}, Silverman proved that, for $d \geq 3$ odd,  the rational map  
$$\phi(z)=i \left(\frac{z-1}{z+1}\right)^{d} \in {\rm Rat}_{d}({\mathbb Q}(i))$$
has field of moduli ${\mathbb Q}$, but it is not definable over it.
These examples are definable over the quadratic extension ${\mathbb Q}(i)$.
Later, in \cite{Hidalgo}, the first author observed that every rational map of odd degree is definable over a suitable quadratic extension of its field of moduli. 
The known explicit examples of rational maps, which are not definable over their field of moduli, were pseudo-real rational maps.
It is natural to ask for the existence of real rational maps which cannot be defined over their field of moduli. In the next result, we provide that such examples exist (it seems to us these are the first kind of such examples).

\begin{theo}\label{ejemplo} 
If $d \geq 3$ is an odd integer, then 
the rational map $$\phi(z)=\left(3+2\sqrt{2}\;\right) \left( \frac{z^{d}+1}{z^{d}-1} \right) \in {\rm Rat}_{d}\left({\mathbb Q}\left(\sqrt{2} \;\right)\right)$$ 
has ${\mathbb Q}$ as its field of moduli but cannot be defined over it.
\end{theo}
\begin{proof}
As ${\rm Gal}\left({\mathbb Q}\left(\sqrt{2}\;\right)/{\mathbb Q}\right)=\langle \sigma \rangle \cong {\mathbb Z}_{2}$, where $\sigma\left(\sqrt{2}\;\right)=-\sqrt{2}$, and  $T(z)=-1/z$ conjugates $\phi$ into $\phi^{\sigma}$, we obtain that ${\mathcal M}_{\phi} ={\mathbb Q}$. Since the critical set of $\phi$ is $C_{\phi}=\{0,\infty\}$, every holomorphic automorphism of $\phi$  must be of the form
$A(z)=\lambda z$ or $B(z)=\mu/z$. As 
$$A \circ \phi \circ A^{-1}(z)=\lambda \left(3+2\sqrt{2}\;\right) \left( \frac{z^{d}+\lambda^{d}}{z^{d}-\lambda^{d}} \right) \quad \mbox{and} \quad 
B \circ \phi \circ B^{-1}(z)=\frac{\mu}{3+2\sqrt{2}} \left( \frac{\mu^{d}-z^{d}}{\mu^{d}+z^{d}} \right),$$
we may see the following facts.
\begin{enumerate}
\item If $A$ is an automorphism of $\phi$, then  
$-\left(3+2\sqrt{2}\;\right)=\phi(0)=A \circ  \phi \circ A^{-1}(0)=-\lambda \left(3+2\sqrt{2}\;\right)$; so $\lambda=1$ and $A(z)=z$.

\item If $B$ is an automorphism of $\phi$, then  
\begin{enumerate}
\item[(i)] $-\left(3+2\sqrt{2}\;\right)=\phi(0)=B\circ  \phi \circ B^{-1}(0)=\dfrac{\mu}{\left(3+2\sqrt{2}\;\right)}$, so $\mu=-\left(3+2\sqrt{2}\;\right)^{2}$; and
\item[(ii)] $\infty=\phi(1)=B \circ \phi \circ B^{-1}(1)=\dfrac{\mu (\mu^{d}-1)}{\left(3+2\sqrt{2}\;\right)\left(\mu^{d}+1\right)}$, so $\mu^{d}=-1$, a contradiction with (i).
\end{enumerate}
\end{enumerate}

The above asserts that ${\rm Aut}(\phi)=\{I\}$ and, in particular, that 
$T$ is the only M\"obius transformation conjugating $\phi$ into $\phi^{\sigma}$. If we assume, by the contrary, that $\phi$ is definable over ${\mathbb Q}$, then there is a M\"obius transformation $L \in {\rm PSL}_{2}({\mathbb C})$ and a rational map $\psi \in {\rm Rat}_{d}\left({\mathbb Q}\right)$ so that $L \circ \phi \circ L^{-1}=\psi$. If  $\tau \in {\rm Gal}\left({\mathbb C}/{\mathbb Q}\left(\sqrt{2}\;\right)\right)$, then 
$$L \circ \phi \circ L^{-1}= \psi=\psi^{\tau}=(L \circ \phi \circ L^{-1})^{\tau}=L^{\tau} \circ \phi^{\tau} \circ (L^{\tau})^{-1}=L^{\tau} \circ \phi \circ (L^{\tau})^{-1},$$ 
so $L^{-1} \circ L^{\tau} \in {\rm Aut}(\phi)=\{I\}$, that is, $L^{\tau}=L$. This asserts that $L \in {\rm PSL}_{2}\left({\mathbb Q}\left(\sqrt{2}\;\right)\right)$.
If we set $f=(L^{\sigma})^{-1} \circ L$, then
$$f \circ \phi \circ f^{-1}=(L^{\sigma})^{-1} \circ L \circ \phi \circ L^{-1} \circ L^{\sigma}=(L^{\sigma})^{-1} \circ \psi \circ L^{\sigma}=(L^{\sigma})^{-1} \circ \psi^{\sigma} \circ L^{\sigma}=(L^{-1} \circ \psi \circ L)^{\sigma} =\phi^{\sigma},$$ 
so $f=T$, i.e., $L=L^{\sigma} \circ T$. Now, if we write 
$$L=\left(\begin{array}{cc}
\alpha & \beta\\
\gamma & \delta 
\end{array}
\right) \in {\rm SL}_{2} \left({\mathbb Q}\left(\sqrt{2}\;\right)\right),
$$
then the equality $L=L^{\sigma} \circ T$ asserts that
$$\left(\begin{array}{cc}
\sigma(\beta) & -\sigma(\alpha)\\
\sigma(\delta) & -\sigma(\gamma) 
\end{array}
\right) =\pm \left(\begin{array}{cc}
\alpha & \beta\\
\gamma & \delta 
\end{array}
\right),
$$
that is, either
$$\left\{\sigma(\beta)=\alpha,\; \sigma(\alpha)=-\beta \right\} \mbox{ or } 
\left\{\sigma(\beta)=-\alpha,\; \sigma(\alpha)=\beta\right\}.$$

In any of the two situations, $\alpha=\beta=0$, a contradiction (for instance, in the first case $\beta=\sigma^{2}(\beta)=\sigma(\alpha)=-\beta$).
\end{proof}

\begin{rema}
In terms of Silverman's cohomology obstructions (see Section \ref{cohomology}), the above proof asserts, for ${\mathbb K}={\mathbb Q}\left(\sqrt{2}\;\right)$ and ${\rm Aut}_{\mathbb K}({\mathbb C})={\rm Gal}\left({\mathbb Q}\left(\sqrt{2}\;\right)/{\mathbb Q}\right) \cong {\mathbb Z}_{2}$, can be stated as follows.
First we prove that, for the generator $\sigma \in {\rm Gal}\left({\mathbb Q}\left(\sqrt{2}\;\right)/{\mathbb Q}\right) \cong {\mathbb Z}_{2}$ there is a unique M\"obius transformation $T$ such that $T \circ \phi \circ T^{-1}=\phi^{\sigma}$. This means that there is exactly one $1$-cocycle for $\phi$
$${\rm Gal}\left({\mathbb Q}\left(\sqrt{2}\;\right)/{\mathbb Q}\right) \to {\rm PSL}_{2}\left({\mathbb Q}\left(\sqrt{2}\right)\right): \sigma \mapsto T.$$ 
Then, we proved that such a  cocycle cannot be a coboundary, obtaining a non-trivial element of the Galois cohomology group ${\rm H}^{1}\!\left({\rm Gal}\left({\mathbb Q}\left(\sqrt{2}\;\right)/{\mathbb Q}\right),{\rm PSL}_{2}\left({\mathbb Q}\left(\sqrt{2}\right)\right)\right)$. 
\end{rema}

\begin{rema}
In Theorem \ref{ejemplo}, we may replace $3+2\sqrt{2}$ by any $h \in {\mathbb Q}\left(\sqrt{2}\;\right)\setminus{\mathbb Q}$ of norm $1$ to obtain similar examples.
\end{rema}

\subsection{Arithmetic real rational maps}\label{Sec:arit}
We say that a subfield ${\mathcal Q}$ of ${\mathbb C}$ is conjugate-invariant if it is invariant under the usual complex conjugation map $J(z)=\overline{z}$. Important cases of conjugate-invariant subfields are the Galois extensions of ${\mathbb Q}$ (the case of arithmetic rational maps).

\begin{theo}\label{teoaritmetico}
Let ${\mathcal Q}$ be a conjugate-invariant algebraically closed subfield of ${\mathbb C}$.
A real rational map which is definable over ${\mathcal Q}$ is also definable over ${\mathbb R} \cap {\mathcal Q}$.
\end{theo} 
\begin{proof}
As rational maps of degree at most one and also those of even degree are definable over their field of moduli \cite{Silv1}, the result holds in those cases,
So, we assume, from now on, that $\phi$ has odd degree $d \geq 3$  and it is defined over ${\mathcal Q}$. As we are assuming it is definable over the reals, it admits a reflection $R$ as antiholomorphic automorphism. As the group of automorphisms of $\phi$ is finite and we are assuming ${\mathcal Q}$ to be algebraically closed, $R$ is necessarily defined over ${\mathcal Q}$. Let us write
$$R(z)=\frac{\alpha \overline{z} + \beta}{\gamma \overline{z} + \delta}, \; \alpha, \beta, \delta, \gamma \in {\mathcal Q}.$$ 

The locus of fixed points of $R$ in the complex plane are the solutions of the equation (which we know is either an Euclidean circle or an Euclidean line)
$\gamma |z|^{2} +\delta z-\alpha \overline{z}-\beta=0.$
If $z=x+iy$, then the above is
$\gamma (x^{2}+y^{2}) +(\delta-\alpha)x +i (\delta+\alpha)y-\beta=0.$

Taking $x \in {\mathcal Q}$, we obtain (at most two and at least one) solutions for $y \in {\mathcal Q}$. It follows that in the locus of fixed points of $R$ there are infinitely many of them inside  ${\mathcal Q}$; take two of them, say $z_{1}$ and $z_{2}$. If we consider the M\"obius transformation $T_{1}(z)=(z-z_{1})/(z-z_{2})$, then $T_{1} \circ R \circ T_{1}^{-1}$ provides a reflection whose locus of fixed points is an Euclidean line together the point at infinity, that is, $T_{1} \circ R \circ T_{1}^{-1}(z)=\omega \overline{z}$, where $\omega \in {\mathcal Q}$ satisfies that $|\omega|=1$.
 If we now consider the rotation $T_{2}(z)=\omega^{1/2} z$, then $T_{2}^{-1} \circ T_{1} \circ R \circ T_{1}^{-1} \circ T_{2}(z)=\overline{z}$. In conclusion, up to conjugation by a suitable M\"obius transformations with coefficients in ${\mathcal Q}$, we may assume that $R=J$. Now, since $J$ is automorphism of $\phi$ we have that $\phi=J \circ \phi \circ J=\overline{\phi}$,
from which we may see that $\phi$ is now defined over ${\mathbb R} \cap {\mathcal Q}$.
\end{proof}

\subsection*{\bf Acknowledgments}
The authors would like to thank the referees for their commments and suggestions which permitted us to improve the paper.


\end{document}